\documentclass[11pt,reqno,tbtags]{amsart}
\usepackage{amssymb}
\usepackage{graphicx}
\usepackage{natbib}
\bibpunct[, ]{[}{]}{;}{n}{,}{,}

\numberwithin{equation}{section}
\allowdisplaybreaks

\newtheorem*{property*}{Property \csname @currentlabel\endcsname}

\newtheorem{theorem}{Theorem}[section]

\newtheorem{proposition}[theorem]{Proposition}
\newtheorem{corollary}[theorem]{Corollary}

\theoremstyle{definition}
\newtheorem{example}[theorem]{Example}
\newtheorem{condition}[theorem]{Condition}

\newtheorem{remark}[theorem]{Remark}

\newtheorem*{ack}{Acknowledgement}

\theoremstyle{remark}

\newenvironment{romenumerate}{\begin{enumerate}
 }{\end{enumerate}}

\newcounter{oldenumi}
{\setcounter{oldenumi}{\value{enumi}}
\begin{romenumerate} \setcounter{enumi}{\value{oldenumi}}}
{\end{romenumerate}}

\newcounter{thmenumerate}

\newcounter{xenumerate}   

\newcommand{\refT}[1]{Theorem~\ref{#1}}

\newcommand{\refR}[1]{Remark~\ref{#1}}
\newcommand{\refS}[1]{Section~\ref{#1}}
\newcommand{\refSS}[1]{Subsection~\ref{#1}}
\newcommand{\refP}[1]{Proposition~\ref{#1}}
\newcommand{\refE}[1]{Example~\ref{#1}}
\newcommand{\refF}[1]{Figure~\ref{#1}}
\newcommand{\refApp}[1]{Appendix~\ref{#1}}
\newcommand{\refCond}[1]{Condition~\ref{#1}}
\newcommand{\refand}[2]{\ref{#1} and~\ref{#2}}

\begingroup
  \divide\count255 by 60
  \multiply\count255 by -60
  \advance\count255 by \time
  \ifnum \count255 < 10 \xdef\klockan{\the\count1.0\the\count255}
  \else\xdef\klockan{\the\count1.\the\count255}\fi
\endgroup

\newcommand{\sumi}{\sum_{i=0}^\infty}
\newcommand{\sumj}{\sum_{j=0}^\infty}

\newcommand{\suml}{\sum_{l=0}^\infty}
\newcommand{\sumlj}{\sum_{l=j}^\infty}
\newcommand{\sumji}{\sum_{j=1}^\infty}
\newcommand{\sumii}{\sum_{i=1}^\infty}

\newcommand{\sumnn}{\sum_{n=-\infty}^\infty}
\newcommand{\sumjj}{\sum_{j=-\infty}^\infty}

\newcommand{\sumin}{\sum_{i=1}^n}

\newcommand\set[1]{\ensuremath{\{#1\}}}

\newcommand\xpar[1]{(#1)}
\newcommand\bigpar[1]{\bigl(#1\bigr)}
\newcommand\Bigpar[1]{\Bigl(#1\Bigr)}

\newcommand\lrpar[1]{\left(#1\right)}

\def\rompar(#1){\textup(#1\textup)}    

\def\xexp(#1){e^{#1}}

\newcommand\ntoo{\ensuremath{{n\to\infty}}}

\newcommand\downto{\searrow}
\newcommand\upto{\nearrow}

\newcommand\ie{i.e.\spacefactor=1000}
\newcommand\eg{e.g.\spacefactor=1000}
\newcommand\viz{viz.\spacefactor=1000}
\newcommand\cf{cf.\spacefactor=1000}
\newcommand{\as}{a.s.\spacefactor=1000}

\newcommand\whp[1]{w.h.p\ifx#1.\relax.\else.\spacefactor=1000#1\fi\relax}

\newcommand\ii{\mathrm{i}}

\newcommand{\tend}{\longrightarrow}
\newcommand\dto{\overset{\mathrm{d}}{\tend}}
\newcommand\pto{\overset{\mathrm{p}}{\tend}}

\newcommand\bbR{\mathbb R}

\newcommand\bbN{\mathbb N}

\newcounter{CC}
\newcounter{cc}

\newcommand\E{\operatorname{\mathbb E{}}}
\renewcommand\P{\operatorname{\mathbb P{}}}

\newcommand\Po{\operatorname{Po}}
\newcommand\Bi{\operatorname{Bi}}

\newcommand\ga{\alpha}
\newcommand\gb{\beta}
\newcommand\gd{\delta}

\newcommand\gf{\varphi}

\newcommand\gl{\lambda}

\newcommand\go{\omega}

\newcommand\gss{\sigma^2}
\newcommand\eps{\varepsilon}

\newcommand\cA{\mathcal A}

\newcommand\cC{\mathcal C}

\newcommand\cE{\mathcal E}

\newcommand\cT{{\mathcal T}}

\newcommand\cX{{\mathcal X}}

\newcommand\ett[1]{\boldsymbol1[#1]} 

\def\[#1]{[\![#1]\!]}

\newcommand\qq{^{1/2}}
\newcommand\qqw{^{-1/2}}
\newcommand\qw{^{-1}}
\newcommand\qww{^{-2}}

\renewcommand{\=}{:=}

\newcommand\intoi{\int_0^1}
\newcommand\intoo{\int_0^\infty}
\newcommand\intoooo{\int_{-\infty}^\infty}
\newcommand\oi{[0,1]}
\newcommand\ooi{(0,1]}
\newcommand\ooo{(0,\infty)}

\newcommand\ddx{\,\textup{d}}

\newcommand{\pgf}{probability generating function}

\newcommand\sumd{\sum_i d_i}

\newcommand\cCi{\cC_1}

\newcommand\opn{o_p(n)}
\newcommand\ds{degree sequence}

\newcommand\pix{\ensuremath{\boldsymbol \pi}}
\newcommand\vpi{_{\pi,\mathsf{v}}}
\newcommand\vpix{_{\pix,\mathsf{v}}}
\newcommand\epi{_{\pi,\mathsf{e}}}
\newcommand\xpi{_{\pi}}

\newcommand\xv{_{\mathsf{v}}}
\newcommand\xe{_{\mathsf{e}}}
\newcommand\xxv[1]{_{#1,\mathsf{v}}}
\newcommand\xxe[1]{_{#1,\mathsf{e}}}
\newcommand\rhov{\rho\xv}
\newcommand\rhoe{\rho\xe}
\newcommand\vpiw{_{\pi',\mathsf{v}}}

\newcommand\dnn{\ensuremath{(d_i\nn)_1^n}}
\newcommand\dd{\ensuremath{\mathbf d}}
\newcommand\gnp{\ensuremath{G(n,p)}}
\newcommand\gnm{\ensuremath{G(n,m)}}
\newcommand\gnln{\ensuremath{G(n,\gl/n)}}
\newcommand\gndq{\ensuremath{G(n,d)}}
\newcommand\gndqx{\ensuremath{G^*(n,d)}}
\newcommand\gnd{\ensuremath{G(n,\dd)}}
\newcommand\gndx{\ensuremath{G^*(n,\dd)}}

\newcommand\nnu{^{(\nu)}}
\newcommand\nn{^{(n)}}
\newcommand\xD{\hat D}
\newcommand\Doo{D}
\newcommand\D{D^*}
\newcommand\nr{n_+}
\newcommand\pixz{p}
\newcommand\cXx{\overline{\mathcal X}}

\newcommand\y{\tilde}
\newcommand\ny{\y n}
\newcommand\nyo{n^\circ}
\newcommand\dy{\y{d}}
\newcommand\ddy{\y{\mathbf{d}}}
\newcommand\py{\y p}
\newcommand\gly{\y\gl}
\newcommand\Dy{\y D}
\newcommand\gdy{g_{\Dy}}
\newcommand\cCy{\y \cC}
\newcommand\ccyi{\cCy_1}
\newcommand\gnxy[1]{\ensuremath{G^*(\ny,#1)}}
\newcommand\gndxy{\gnxy{\ddy}}
\newcommand\ddyv{\ddy\vpi}
\newcommand\ddyvx{\ddy\vpix}
\newcommand\ddye{\ddy\epi}
\newcommand\hy{\y{h}}

\newcommand\z[1]{#1'}
\newcommand\nz{\z n}
\newcommand\ddz{\z{\mathbf{d}}}
\newcommand\gndxz{\ensuremath{G^*(\nz,\ddz)}}

\newcommand\zooo{_0^\infty}
\newcommand\nrset[1]{\#\set{#1}}

\newcommand\piqq{\pi\qq}
\newcommand{\bxxx}[3]{b_{#1#2}(#3)}
\newcommand{\bljpiqq}{\bxxx lj{\piqq}}
\newcommand{\blipiqq}{\bxxx l1{\piqq}}
\newcommand{\bljp}{\bxxx ljp}
\newcommand\gdoo{g_{\Doo}}
\newcommand\gdp{g_{\y\Doo}}
\newcommand\vz{\chi}
\newcommand\ez{\mu}
\newcommand\pic{\pi_{\mathsf c}}
\newcommand\pick{\pic^{(k)}}
\newcommand\picii{\pic^{(2)}}
\newcommand\picn{\pi_{\mathsf c n}}

\newcommand\crk{\ensuremath{\mathrm{Core}_k}}
\newcommand\crkx{\ensuremath{\mathrm{Core}^*_k}}
\newcommand\x{,\allowbreak}
\newcommand\pmax{\widehat p}
\newcommand\pmaxe{\widehat p_0}
\newcommand\Dpmax{D_{\pmax}}
\newcommand\Dypmax{\Dy_{\pmax}}
\newcommand\Dypmaxe{\Dy_{\pmaxe}}
\newcommand\kcore{$k$-core}
\newcommand\xiv{\xi\xv}

\newcommand\hxi{\hat\xi}
\newcommand\tk{\cT_k}
\newcommand\tkx{\overline{\cT}_k}

\newcommand\tkn{\cT_{kn}}
\newcommand\hpsi{\widehat\psi}
\newcommand\lmp{local maximum point}
\newcommand\gmp{global maximum point}
\newcommand\crp{critical point}
\newcommand\tp{\tilde p}
\newcommand\tpi{\tilde \pi}
\newcommand\phtr{phase transition}
\newcommand\gfoo{\gf_\infty}
\newcommand\ppp{\tp_0}
\newcommand\ppl{\tp_1}
\newcommand\iix{\cA}
\newcommand\iio{\iix_0}
\newcommand\iif{\iix_f}
\newcommand\iifl{\iix_f^{(\ell)}}
\newcommand\qc{q_{\mathsf c}}
\newcommand\qcl{q_{\mathsf c}^{(\ell)}}
\newcommand\bgf{\bar\gf}

\newcommand{\maple}{\texttt{Maple}}

\newcommand\REM[1]{{\raggedright\texttt{[#1]}\par\marginal{XXX}}}

\newcommand\urladdrx[1]{{\urladdr{\def~{{\tiny$\sim$}}#1}}}

\begin{document}
\title[percolation in random graphs with given degrees]
{On percolation in random graphs with given vertex degrees}

\date{April 10, 2008} 

\author{Svante Janson}
\address{Department of Mathematics, Uppsala University, PO Box 480,
SE-751~06 Uppsala, Sweden}
\email{svante.janson@math.uu.se}
\urladdrx{http://www.math.uu.se/~svante/}

\keywords{random graph, giant component, k-core, bootstrap percolation}
\subjclass[2000]{60C05; 05C80} 

\begin{abstract} 
We study the random graph obtained by random deletion of vertices or
edges from a random graph with given vertex degrees. A simple trick of
exploding vertices instead of deleting them, enables us to derive
results from known results for random graphs with given vertex
degrees.
This is used to study existence of giant component and existence of
\kcore. As a variation of the latter, we study also bootstrap
percolation in random regular graphs.

We obtain both simple new proofs of known results and new results.
An interesting feature is that for some degree sequences, there are
several or even infinitely many phase transitions for the \kcore.
\end{abstract}

\maketitle

\section{Introduction}\label{Sintro}

One popular and important type of random graph is given by the 
uniformly distributed 
{random graph with a given degree sequence},
defined as follows.
Let $n \in \bbN$ and let $\dd=(d_i)_1^n$ be a sequence of
non-negative integers. 
We let \gnd{} be a 
random graph with degree sequence 
$\dd$, uniformly chosen among all possibilities (tacitly
assuming that there is any such graph at all; in particular, $\sumd$
has to be even).

It is well-known that it is often simpler to study the corresponding
random multigraph \gndx{}
with given degree sequence $\dd=(d_i)_1^n$, defined
for every sequence $\dd$ with $\sum_id_i$ even
by the configuration model (see \eg{} \citet{Bollobas}): 
take a set of $d_i$ \emph{half-edges} for each vertex 
$i$, and combine the half-edges into pairs by a uniformly random
matching of the set of all half-edges
(this pairing  is called a \emph{configuration}); each pair of
half-edges is then joined to form an edge of \gndx.

We consider asymptotics as the numbers of vertices tend to infinity,
and thus we assume throughout the paper that we are given, for each $n$,
a sequence
$\dd\nn=(d_i\nn)_1^{n}$ 
with $\sum_i d_i\nn$ even.
(As usual, we could somewhat more generally assume that we are given a sequence
$n_\nu\to\infty$ and for each $\nu$ a sequence
$\dd\nnu=(d_i\nnu)_1^{n_\nu}$.) 
For notational simplicity we will usually not show the dependency on
$n$ explicitly; 
we thus write $\dd$ and $d_i$,
and similarly for other (deterministic or random) quantities
introduced below.
All unspecified limits and other asymptotic statements are for
$n\to\infty$.
For example, \whp{} (with high probability) means 'with probability
tending to 1 as \ntoo', and
$\pto$ means 'convergence in probability as \ntoo'.
Similarly, we use 
$o_p$ and 
$O_p$ in the standard way, 
always implying \ntoo.
For example, if $X$ is a parameter of the random graph,
$X=o_p(n)$ means that $\P(X>\eps n)\to0$ as \ntoo{} for every $\eps>0$; 
equivalently, 
$X/n\pto0$.

We may obtain \gnd{} by
conditioning the multigraph \gndx{} on being a (simple) graph,
\ie, on not having any
multiple edges or loops.
By \citet{SJ195} (with earlier partial results by many authors),
\begin{equation}
  \label{simple}
\liminf\P\bigpar{\gndx \text{ is simple}}>0
\iff
\sumin d_i^2 = O\lrpar{\sumin d_i}.
\end{equation}
In this case, many results transfer immediately from \gndx{} to \gnd,
for example, every result of the type $\P(\cE_n)\to0$ for some events
$\cE_n$, and thus every result saying that some parameter converges in
probability to some non-random value.
This includes every result in the present paper.

We will in this paper study the random multigraph \gndx; the reader
can think of doing this either for its own sake or as a tool for
studying \gnd. We leave the statement of corollaries for \gnd{}, using
\eqref{simple}, to the reader. 
Moreover, the results for \gnd{} extend to some other random graph
models too, in particular $\gnp$ with $p\sim \gl/n$  and $\gnm$
with $m\sim \gl n/2$ with $\gl>0$, by the standard device of
conditioning on the degree sequence; again we omit the details and
refer to \cite{SJ184,SJ196,SJ204} where this method is used.

We will consider percolation of these random (multi)graphs,
where we first generate a random graph \gndx{}
and then delete either vertices
or edges at random. 
(From now on, we simply write 'graph' for 'multigraph'.)
The methods below can be combined to treat the
case of 
random deletion of both vertices and edges, which is
studied by other methods in \eg{} \citet{SJ199},
but we leave this to the reader. 

To be precise, 
we consider the
following two constructions, given any graph $G$ and a probability
$\pi\in[0,1]$.
\begin{description}
  \item[Site percolation]
Randomly delete each vertex (together with all incident edges) 
with probability $1-\pi$, 
independently of all other vertices.
We denote the resulting random graph by $G\vpi$.
  \item[Bond percolation]
Randomly delete each edge with
probability $1-\pi$, independently of all other edges.
(All vertices are left.)
We denote the resulting random graph by $G\epi$.
\end{description}
Thus $\pi$ denotes the probability to be kept in the percolation
model.
When, as in our case, the original graph $G$ itself is random, it is further
assumed that 
we first sample $G$ and then proceed as above, 
conditionally on $G$.

The cases $\pi=0,1$ are trivial: 
$G\xxv1=G\xxe1=G$, while
$G\xxv0=\emptyset$, the totally empty graph with no vertices and no edges, and
$G\xxe0$ is the empty graph with the same edge set as $G$ but
no edges. We will thus mainly consider $0<\pi<1$.

We may generalize the site percolation model by letting the
probability depend on the degree of the vertex. Thus, if 
$\pix=(\pi_d)_0^\infty$ is a given sequence of probabilities $\pi_d\in\oi$,
let $G\vpix$ be the random graph obtained by deleting vertices
independently of each other, with vertex
$v\in G$ deleted with probability $1-\pi_{d(v)}$ where $d(v)$ is the degree of
$v$ in $G$.

For simplicity and in order to concentrate on the main ideas,
we will in this paper consider only the case when the
probability $\pi$ (or the sequence $\pix$) is fixed and thus does not
depend on $n$, with the exception of a few remarks where we briefly
indicate how the method can be used also for a more detailed study of
thresholds.

The present paper is inspired by \citet{Fount},
and we follow his idea of deriving results for the percolation models
$\gndx\vpix$ and $\gndx\epi$ from results for the model
$\gndx$ without deletions, but for different \ds{s} $\dd$.
We will, however, use another method to do this, which we find simpler.

\citet{Fount} shows that for both site and bond percolation on \gndx, if we
condition the resulting random graph on its degree sequence $\ddz$,
and let $\nz$ be the number of its vertices, then
the graph has the 
distribution of $\gndxz$, the random graph with this degree sequence
constructed by the configuration model. He then proceeds to calculate
the distributions of the degree sequence $\ddz$ for the two
percolation models and finally applies known results to $\gndxz$.

Our method is a version of this, where we do the deletions in two steps.
For site percolation, instead of deleting a vertex, let us
first \emph{explode} it by replacing it  
by $d$ new vertices of degree 1, where $d$ is its degree;
we further colour the new vertices \emph{red}.
Then clean up by removing all red vertices.
Note that the (random) explosions change the number of vertices, but
not the number of half-edges.
Moreover, given the set of explosions, there is a one-to-one
correspondence between configurations before and after the explosions,
and thus, if we condition on the new degree sequence,
the exploded graph is still described by the configuration model.
Furthermore, by symmetry, when removing the red vertices, all vertices of
degree 1 are equivalent, so we may just as well remove the right
number of vertices of degree 1, but choose them uniformly at random.
Hence, we can obtain $\gndx\vpi$ as follows:
\begin{description}
  \item[Site percolation]
For each vertex $i$, replace it 
with probability $1-\pi$
by $d_i$ new vertices of degree $1$ (independently of all other vertices).
Let $\ddyv$ be the resulting (random) degree sequence, 
let $\ny$ be its length (the number of vertices),
and let $\nr$ be the
number of new vertices. Construct the random graph $\gnxy{\ddyv}$. Finish by
deleting $\nr$ randomly chosen vertices of degree $1$.
\end{description}

The more general case when we are given a sequence
$\pix=(\pi_d)_0^\infty$ is handled in  the same way:

\begin{description}
  \item[Site percolation, general]
For each vertex $i$, replace it 
with probability $1-\pi_{d_i}$
by $d_i$ new vertices of degree $1$.
Let $\ddyvx$ be the resulting (random) degree sequence, 
let $\ny$ be its length (the number of vertices),
and let $\nr$ be the
number of new vertices. Construct the random graph $\gnxy{\ddyvx}$. Finish by
deleting $\nr$ randomly chosen vertices of degree $1$.
\end{description}

\begin{remark}\label{Rfixed}
  We have here assumed that vertices are deleted at random,
  independently of each other. This is not essential for our method,
  which may be further extended to the case when we remove a set
  of vertices determined by any random procedure that is independent
  of the edges in \gnd{} (but may depend on the vertex degrees).
For example, we may remove a fixed number $m$ of vertices, chosen
  uniformly at random. It is easily seen that if $m/n\to\pi$, the
  results of \refSS{SSvertex} below still hold (with all $\pi_j=\pi$), 
and thus the results of the later sections hold too.
Another, deterministic, 
example is to remove the first $m$ vertices.
\end{remark}

For bond percolation, we instead explode each half-edge
with probability $1-\sqrt \pi$, independently of all other half-edges;
to explode a half-edge means that we
disconnect it from its vertex and
transfer it to a new, red 
vertex of degree 1. Again this does not change the number of
half-edges, and there is a one-to-one correspondence between
configurations before and after the explosions. We finish by removing
all red vertices and their incident edges. Since an edge consists of
two half-edges, and each survives with probability $\sqrt \pi$, this
gives the bond percolation model $\gndx\epi$ where edges are kept with
probability $\pi$.
This yields the following recipe:

\begin{description}
  \item[Bond percolation]
Replace the degrees $d_i$ in the sequence $\dd$ by independent random
degrees $\dy_i\sim\Bi(d_i,\sqrt \pi)$. Add $\nr\=\sumin (d_i -\dy_i)$ new
degrees 1 to the sequence $(\dy_i)_1^n$, and
let $\ddye$ be the resulting degree sequence
and $\ny=n+\nr$ its length.
Construct the random graph $\gnxy{\ddye}$. Finish by
deleting $\nr$ randomly chosen vertices of degree $1$.
\end{description}

In both cases, we have reduced the problem to a simple (random)
modification of the \ds, plus a random removal of a set of vertices of
degree 1. 
The latter is often more or less trivial to handle, see the
applications below. 
We continue to call the removed vertices red when convenient.

Of course, to use this method, it is essential to find the degree
sequence $\ddy$ after the explosions. We study this in \refS{Sdegrees}.
We then apply this method to three different problems:

\emph{Existence of a giant component} in the percolated graph, \ie, what
is called \emph{percolation} in random graph theory
(\refS{Sgiant}). Our results include
and extend earlier work by \citet{Fount}, which inspired the present study,
and some of the results by \citet{SJ199}.

\emph{Existence of a $k$-core} in the percolated graph (\refS{Score}).
We obtain a general result analogous to
(and extending) the well-known result by
\citet{PSW} for $\gnp$. We study the phase transitions that may occur
in some detail and show by examples that it is possible to have
several, and even an infinite number of, different phase
transitions as the probability $\pi$ increases from 0 to 1.

\emph{Bootstrap percolation in random regular graphs} 
(\refS{Sbootstrap}), where we obtain a new and simpler
proof of results by 
\citet{BalPitt}.

For a graph $G$, let 
$v(G)$ and $e(G)$ denote the numbers of  vertices and edges in $G$,
respectively, and let $v_j(G)$ be the number of vertices of
degree $j$, $j\ge0$.
We sometimes use $\gndx\xpi$ to denote any of the percolation
models $\gndx\vpi$, $\gndx\vpix$ or $\gndx\epi$.

\section{The degree sequence after explosions}\label{Sdegrees}

Let $n_j\=\nrset{i\le n:d_i=j}$, the number 
$v_j(\gndx)$
of vertices of degree $j$
in \gndx. Thus $\sumj n_j=n$.
We assume for simplicity the following regularity condition.

\begin{condition}\label{C1}
There exists
a probability distribution 
$(p_j)_{j=0}^\infty$ with finite positive mean
$\gl\=\sum_j j p_j\in(0,\infty)$
such that (as \ntoo)
\begin{equation}
  \label{njlim}
n_j/n\to\pixz_j,
\qquad j\ge0,
\end{equation}
and 
\begin{equation}
  \label{njsum}
\frac{\sumj jn_j}{n}\to\gl\=\sumj j\pixz_j.
\end{equation}
\end{condition}
Note that, in order to avoid trivialities, we assume that 
$\gl>0$, which is equivalent to
$p_0<1$.
Thus, there is a positive fraction of vertices of degree at least 1.

Note that $\sum_j jn_j=\sum_i d_i$ equals twice the number of edges in
\gndx, 
and that \eqref{njsum} says that the average degree in \gndx{} converges
to $\gl$.

Let the random variable $\xD=\xD_n$ 
be the degree of a random vertex in \gndx{}, thus $\xD_n$ has
the distribution $\P(\xD_n=j)=n_j/n$, and let
$\Doo$ be a random variable with the distribution $(\pixz_j)\zooo$.
Then \eqref{njlim} is equivalent to
$\xD_n\dto \Doo$, and \eqref{njsum} is $\E \xD_n\to\gl=\E\Doo$.
Further, assuming \eqref{njlim}, 
\eqref{njsum} is equivalent to uniform integrability of $\xD_n$, or
equivalently uniform
summability (as \ntoo) 
of $\sum_j jn_j/n$, see for example \citet[Theorem 5.5.9 and Remark
  5.5.4]{Gut}. 

\begin{remark}
  \label{Ropn}
The uniform summability of $\sum_j jn_j/n$ is easily seen to imply
that if $H$ is any (random or deterministic) subgraph on $\gndx$ with
$v(H)=o(n)$, then $e(H)=o(n)$, and similarly with $\opn$.
\end{remark}

We will also use the \pgf{} of the asymptotic degree
distribution $\Doo$:
\begin{equation}
  \label{g}
g_{\Doo}(x)\=\E x^{\Doo}=\sumj p_j x^j,
\end{equation}
defined at least for $|x|\le1$.

We perform either site or bond percolation as in \refS{Sintro}, 
by the explosion method described there,
and
let $\ny_j\=\nrset{i\le \ny:\dy_i=j}$ be the number of
vertices of degree $j$ after the explosions. Thus 
\begin{equation}
  \label{ny}
\sumj \ny_j=\ny.
\end{equation}


It is  easy to find the distribution of $(\ny_j)$ and its
asymptotics for our two percolation models.

\subsection{Site percolation}\label{SSvertex}
We treat the general version with a sequence $\pix$.
Let $\nyo_j$ be the number of vertices of degree $j$ that are not
exploded.
Then
\begin{align}
  \nyo_j&\sim\Bi(n_j,\pi_j)\quad
\text{(independent of each other)},
\\
\nr &= \sumj j(n_j-\nyo_j),
\\
\ny_j &= \nyo_j, \quad j\neq 1,
\\
\ny_1&= \nyo_1+\nr.
\end{align}
By the law of large numbers, $\nyo_j=n_j\pi_j+o_p(n)$ and thus, using the
assumption \eqref{njlim} and the uniform summability of $\sum_j jn_j/n$
(which enables us to treat the infinite sums in \eqref{nrv} and
\eqref{nyv} by a standard argument),
\begin{align}
  \nyo_j&=n_j\pi_j+\opn=\pi_j\pixz_j n+\opn,
\label{n0jv}
\\
\nr &= \sumj j(1-\pi_j)\pixz_jn+\opn,
\label{nrv}
\\
\ny_j &= \pi_j\pixz_j n+\opn, \quad j\neq 1,
\label{nyjv}
\\
\ny_1&= \Bigpar{\pi_1\pixz_1+\sumj j(1-\pi_j)\pixz_j} n+\opn,
\label{ny1v}
\\
\ny&= \sumj \bigpar{\pi_j+j(1-\pi_j)}\pixz_j n+\opn.
\label{nyv}
\end{align}

We can write \eqref{nyv} as
\begin{equation} \label{zetav}
 \frac{\ny}n\pto
\zeta:=\sumj \bigpar{\pi_j+j(1-\pi_j)}\pixz_j>0.
\end{equation}
Further, by \eqref{nyjv} and \eqref{ny1v},
\begin{equation}  \label{nyjlimv}
  \frac{\ny_j}{\ny}
\pto
\py_j\=
\begin{cases}
  \zeta\qw\pi_j p_j,& j\neq1, \\
  \zeta\qw \bigpar{\pi_1\pixz_1+\sumji j(1-\pi_j)\pixz_j}, & j=1.
\end{cases}
\end{equation}
Since $\ny_j\le n_j$ for $j\ge2$ and $\ny\ge n-n_0$, the uniform
summability of $jn_j/n$ implies uniform summability of $j\ny_j/\ny$, and
thus also
\begin{equation}
  \label{nyjsum}
\frac{\sumj j\ny_j}{\ny}
\pto
\gly\=
\sumj j\py_j
<\infty.
\end{equation}
Hence \refCond{C1} holds, in probability, for the random degree sequence
$\ddy$ too.
Further,
the total number of half-edges is not changed by the
explosions, and thus also, by \eqref{zetav} and \eqref{njsum},
\begin{equation}
  \label{nyjsum1}
\frac{\sumj j\ny_j}{\ny}
=
\frac{\sumj jn_j}{\ny}
=
\frac{n}{\ny}\cdot
\frac{\sumj jn_j}{n}
\pto
\zeta\qw\gl;
\end{equation}
hence (or by \eqref{nyjlimv}),
\begin{equation}\label{gly}
  \gly=\zeta\qw\gl.
\end{equation}

In the proofs below it will be convenient to assume that
\eqref{nyjlimv} and   \eqref{nyjsum} hold \as, and not just in
probability, so that \refCond{C1} \as{} holds for $\ddy$; we can
assume this without loss of generality by the
Skorohod coupling theorem \cite[Theorem 4.30]{Kallenberg}.
(Alternatively, one can argue by selecting suitable subsequences.)

Let $\Dy$ have the probability distribution $(\py_j)$, and let $\gdp$
be its \pgf. Then, by \eqref{nyjlimv},
\begin{equation}\label{gdy}
 \zeta \gdp(x)=\sumj \zeta \py_j x^j
=\sumj \pi_jp_j x^j + \sumj j(1-\pi_j)p_j x
=\gl x+\sumj \pi_jp_j (x^j-jx).
\end{equation}
In particular, if all $\pi_j=\pi$,
\begin{equation}\label{gdya}
 \zeta \gdp(x)
=\pi \gdoo(x) + (1-\pi)\gl x,
\end{equation}
where now $\zeta=\pi+(1-\pi)\gl$.

\subsection{Bond percolation}\label{SSedge}

For bond percolation, 
we have  explosions that do not destroy
the vertices, but they may reduce their degrees.
Let $\ny_{jl}$ be the number of vertices that had degree $l$ before the
explosions and $j$ after. Thus $\ny_j=\sum_{l\ge j} \ny_{jl}$ for
$j\neq1$ and $\ny_1=\sum_{l\ge 1} \ny_{1l}+\nr$.
A vertex of degree $l$ will after the explosions have a degree with
the binomial distribution $\Bi(l,\piqq)$, and thus the probability that
it will become a vertex of degree $j$ is the binomial probability
$\bljpiqq$, where we define
\begin{equation}\label{bi}
 \bljp\=
\P\bigpar{\Bi(l,p)=j}
=
\binom lj p^{j}(1-p)^{l-j}.
\end{equation}
Since explosions at different vertices occur independently, this
means that, for $l\ge j\ge 0$,
\begin{equation*}
\ny_{jl}\sim\Bi\bigpar{n_l,\bljpiqq}
\end{equation*}
and thus, by the law of large numbers and \eqref{njlim},
\begin{equation*}
\ny_{jl}=
\bljpiqq p_l n+\opn.
\end{equation*}
Further, the number $\nr$ of new vertices equals the number of
explosions, and thus has the binomial distribution $\Bi(\sum_l ln_l,1-\piqq)$.
Consequently, using 
also \eqref{njsum} and the uniform summability of $\sum_j jn_j/n$,
\begin{align}
\nr &
=\sum_l ln_l(1-\piqq)+\opn
=(1-\piqq)\gl n+\opn,
\label{nre}
\\
\ny_j &= \sum_{l\ge j} \ny_{jl}
=
\sum_{l\ge j} \bljpiqq p_l n+\opn,
\quad j\neq 1,
\label{nyje}
\\
\ny_1 &= \sum_{l\ge 1} \ny_{1l}+\nr
=
\sum_{l\ge 1} \blipiqq p_l n
+\bigpar{1-\piqq}\gl n+\opn,
\label{ny1e}
\\
\ny&
= n+\nr
= 
n+\bigpar{1-\piqq}\gl n+\opn. \label{nye}
\end{align}

In analogy with site percolation we thus have, by \eqref{nye},
\begin{equation} \label{zetae}
 \frac{\ny}n\pto
\zeta:=1+\bigpar{1-\piqq}\gl
\end{equation}
and further, by \eqref{nyje} and \eqref{ny1e},
\begin{equation}  \label{nyjlime}
  \frac{\ny_j}{\ny}
\pto
\py_j\=
\begin{cases}
  \zeta\qw \sum_{l\ge j} \bljpiqq p_l,
   & j\neq1, \\
  \zeta\qw 
\Bigpar{
\sum_{l\ge 1} \blipiqq p_l +\bigpar{1-\piqq}\gl},
& j=1.
\end{cases}
\end{equation}
Again, the uniform
summability of $jn_j/n$ implies uniform summability of $j\ny_j/\ny$, and
the total number of half-edges is not changed;
thus \eqref{nyjsum}, \eqref{nyjsum1} and \eqref{gly} hold, now with
$\zeta$ given by \eqref{zetae}. 
Hence \refCond{C1} holds in probability for the degree sequences
$\ddy$ in bond percolation too,
and by
the Skorohod coupling theorem we may assume that
it holds
a.s.

The formula for $\py_j$ is a bit complicated, but there is a simple
formula for the \pgf{} $\gdp$. We have, by the binomial theorem,
$\sum_{j\le l} \bxxx lj{\pi} x^j = (1-\pi + \pi x)^l$, and thus
\eqref{nyjlime} yields
\begin{equation}\label{gdye}
  \begin{split}
\zeta\gdp(x)
&= \suml (1-\piqq+\piqq x)^l p_l + (1-\piqq)\gl x
\\&
= \gdoo (1-\piqq+\piqq x) + (1-\piqq)\gl x.
  \end{split}
\end{equation}

\section{Giant component}\label{Sgiant}

The question of existence of a giant component
in  \gnd{} and \gndx{} was answered by
\citet{MR95}, who showed that 
(under some weak technical assumptions)
a giant component
exists \whp{}
if and only if (in the notation above) $\E \Doo(\Doo-2)>0$.
(The term \emph{giant component} is in this paper used, somewhat
informally, 
for a component containing at least a fractions $\eps$ of all vertices,
for some small $\eps>0$ that does not depend on $n$.)
They further gave a formula for the size of this giant component in
\citet{MR98}.
We will use the following version of their result, given by
\citet[Theorem 2.3 and Remark 2.6]{SJ204}.
Let, for any graph $G$, $\cC_k(G)$ denote the $k$:th largest component
of $G$. (Break ties by any rule. If there are fewer that $k$
components, let $\cC_k\=\emptyset$.)

\begin{proposition} [\cite{MR98,SJ204}]
  \label{PT1}
Consider \gndx, assuming that \refCond{C1} holds and $p_1>0$.
Let\/ $\cC_k\=\cC_k(\gndx)$
and let
$\gdoo(x)$ be the \pgf{} in \eqref{g}.
  \begin{romenumerate}
\item\label{PT1a}
If\/ $\E \Doo(\Doo-2)=\sum_j j(j-2)p_j>0$, then there is a unique $\xi\in(0,1)$
such that $g_{\Doo}'(\xi)=\gl\xi$,
and 
\begin{align}\label{t1v}
  v(\cC_1)/n&\pto 1-g_{\Doo}(\xi)>0,
\\
  v_j(\cC_1)/n&\pto p_j(1-\xi^j), \text{ for every } j\ge0,
\label{t1vj}
\\
  e(\cC_1)/n&\pto \tfrac12\gl(1-\xi^2).
\label{t1e}
\end{align}
Furthermore,
$ v(\cC_2)/n\pto 0$ and  $ e(\cC_2)/n\pto 0$.
\item\label{PT1b}
If\/ $\E \Doo(\Doo-2)=\sum_j j(j-2)p_j\le0$, then 
$ v(\cC_1)/n\pto 0$ and $e(\cC_1)/n\pto 0$.
  \end{romenumerate}
\end{proposition}

\begin{remark}
  $\E D^2=\infty$ is allowed in \refP{PT1}\ref{PT1a}.
\end{remark}

\begin{remark}\label{RPT1ii}
In \refP{PT1}\ref{PT1b}, where $\E D(D-2)\le0$ and $p_1<0$, for $0\le x<1$
\begin{equation*}
  \begin{split}
\gl x- g_{D}'(x)
&=\sumji jp_j(x-x^{j-1})
=p_1(x-1)+x\sum_{j=2}^\infty jp_j(1-x^{j-2})
\\&
\le p_1(x-1)+x\sum_{j=2}^\infty jp_j(j-2)(1-x)
\\&
<\sumji j(j-2)p_jx(1-x)
=\E D(D-2)x(1-x)\le0.
  \end{split}
\end{equation*}
Hence, in this case the only solution in $[0,1]$ to $g_D'(\xi)=\gl
 \xi$ is $\xi=1$, which we may take as the definition in this case.
\end{remark}

\begin{remark}\label{Rbp}
  Let $\D$ be a random variable with the distribution 
  \begin{equation*}
\P(\D=j)=(j+1)\P(D=j+1)/\gl, \qquad j\ge0;
  \end{equation*}
this is the size-biased
distribution of $D$ shifted by 1, and it has a well-known natural
interpretation as follows.
Pick a random half-edge; then the number of remaining half-edges at
its endpoint has asymptotically the distribution of $\D$.
Therefore, the natural (Galton--Watson) branching process approximation of the
exploration of the successive neighbourhoods of a given vertex
is the branching process $\cXx$ with offspring distributed as $\D$, 
but starting with an initial distribution given by $D$. 
Since
\begin{equation*}
  g_{\D}(x)=\sumji\P(\D=j-1)x^{j-1}
=\sumji\frac{jp_j}{\gl}x^{j-1} 
=\frac{g_D'(x)}{\gl},
\end{equation*}
the equation $g_D'(\xi)=\gl\xi$ in \refP{PT1}\ref{PT1a} can be written
$g_{\D}(\xi)=\xi$, which shows that $\xi$ has an interpretation as the
extinction probability of the branching process $\cX$ with offspring
distribution $\D$, now starting with a single individual.
(This also agrees with the definition in \refR{RPT1ii} for the 
case \refP{PT1}\ref{PT1b}.)
Thus $g_D(\xi)$ in \eqref{t1v} is the extinction probability of $\cXx$.
Note also that  
\begin{equation*}
  \E\D=\frac{\E D(D-1)}{\gl}
=\frac{\E D(D-1)}{\E D},
\end{equation*}
so the condition $\E D(D-2)>0$, or equivalently $\E D(D-1)>\E D$, is
equivalent to $\E\D>1$, the classical condition for the branching
process to be supercritical and thus have a positive survival
probability.

The intuition behind the branching process approximation of the local
structure of a random graph at a given vertex is that an infinite approximating
branching process corresponds to the vertex being in a giant
component. This intuition agrees also with the formulas \eqref{t1vj}
and \eqref{t1e}, which reflect the fact that a vertex of degree $j$
[an edge]  belongs to the giant component if and only if one of its
$j$ attached half-edges [one of its two constituent half-edges] connects
to the giant component.
(It is rather easy to base rigorous proofs on the branching process
approximation, see \eg{} \cite{SJ199}, but in the present paper we
will only use the branching process heuristically.)
\end{remark}

Consider one of our percolation models $\gndx\xpi$, and construct it using
explosions and an intermediate random graph $\gndxy$ as
described in the introduction.
(Recall that $\ddy$ is random, while $\dd$ and the limiting
probabilities $p_j$ and $\py_j$ are not.)
Let $\cC_j\=\cC_j\bigpar{\gndx\xpi}$ 
and $\cCy_j\=\cC_j\bigpar{\gndxy}$ 
denote the components of $\gndx\xpi$,
and $\gndxy$, respectively.

As remarked in \refS{Sdegrees}, 
we may assume that 
$\gndxy$ too satisfies \refCond{C1}, with $p_j$ replaced by $\py_j$.
(At least a.s.; recall that $\ddy$ is random.)
Hence, assuming $\py_1>0$, 
if we first condition on $\ddy$, then
\refP{PT1} applies immediately to the exploded graph $\gndxy$. We
also have to remove $\nr$ randomly chosen ``red'' vertices of degree
1, but luckily this will not break up any component.
Consequently, if $\E\Dy(\Dy-2)>0$, then $\gndxy$ \whp{} has a giant
component $\cCy_1$, with $v(\cCy_1)$, $v_j(\cCy_1)$ and $e(\cCy_1)$
given by \refP{PT1}
(with $p_j$ replaced by $\py_j$), and 
after removing the red vertices, the remainder of $\ccyi$
is still connected and forms a  component $\cC$ in $\gndx\xpi$.
Furthermore, 
since $\E\Dy(\Dy-2)>0$,
$\py_j>0$ for at least one $j>2$, and it follows by
\eqref{t1vj} that $\ccyi$ contains $cn+\opn$ vertices of degree $j$,
for some $c>0$; all these belong to $\cC$ (although possibly with
smaller degrees), so $\cC$ contains \whp{} at least $cn/2$
vertices. Moreover, all other components of $\gndx\xpi$ are contained
in components of $\gndxy$ different from $\ccyi$, and thus at most as
large as $\cCy_2$, 
which by \refP{PT1} has
$o_p(\ny)=\opn$ vertices. Hence, \whp{}
$\cC$ is the largest component
$\cC_1$ of $\gndx\xpi$, and this is the unique giant component in
$\gndx\xpi$. 

Since we remove a fraction $\nr/\ny_1$ of all vertices of degree 1,
we remove by the law of large numbers (for a hypergeometric
distribution) about the same fraction of the vertices of degree 1 in
the giant component $\cCy_1$. 
More precisely, by \eqref{t1vj}, $\cCy_1$ contains about
a fraction $1-\xi$ of all vertices of degree 1,
where $\gdy'(\xi)=\gly\xi$; hence the number of
red vertices removed from $\ccyi$ is
\begin{equation}
  \label{arn}
(1-\xi)\nr+o_p(n).
\end{equation}
By \eqref{t1v} and \eqref{arn},
\begin{equation}\label{arnv}
  v(\cCi)=v(\ccyi)-(1-\xi)\nr+\opn
=\ny\bigpar{1-\gdy(\xi)}-\nr+\nr\xi+\opn.
\end{equation}
Similarly, by \eqref{t1e} and \eqref{arn}, since each red vertex that
is removed from $\cCi$ also removes one edge with it,
\begin{equation}\label{arne}
  e(\cCi)=e(\ccyi)-(1-\xi)\nr+\opn
=\tfrac12\gly\ny(1-\xi^2)-(1-\xi)\nr+\opn.
\end{equation}

The case $\E\Dy(\Dy-2)\le0$ is even simpler; since the largest
component $\cC_1$ is contained in some component $\cCy_j$ of $\gndxy$,
it follows that $v(\cC_1)\le v(\cCy_j)\le v(\cCy_1)=o_p(\ny)=\opn$. 

This leads to the following results, where we treat site and bond
percolation separately and add formulas for the asymptotic size of
$\cC_1$.

\begin{theorem}
  \label{TGv}
Consider the site percolation model $\gndx\vpix$, and
suppose that \refCond{C1} holds and
that $\pix=(\pi_d)_0^\infty$ with $0\le\pi_d\le1$; suppose further
that there exists $j\ge1$ such that $p_j>0$ and $\pi_j<1$.
Then there is \whp{} a giant component if and only if 
\begin{equation}
  \label{tgv}
\sumj j(j-1)\pi_jp_j > \gl\=\sumj jp_j.
\end{equation}
  \begin{romenumerate}
\item\label{TGva}
If \eqref{tgv} holds,
then there is a unique $\xi=\xi\xv(\pix)\in(0,1)$
such that 
\begin{equation}\label{tgvxi}
  \sumji j \pi_j p_j(1-\xi^{j-1})=\gl(1-\xi)
\end{equation}
and then
\begin{align}\label{tgvv}
  v(\cC_1)/n&\pto 
\vz\xv(\pix)\=\sumji\pi_jp_j(1-\xi^j) >0,
\\
  e(\cC_1)/n&\pto 
\ez\xv(\pix)\=(1-\xi)\sumji j\pi_jp_j-\frac{(1-\xi)^2}{2}\sumji jp_j.
\label{tgve}
\end{align}
Furthermore,
$ v(\cC_2)/n\pto 0$ and  $ e(\cC_2)/n\pto 0$.
\item\label{TGvb}
If\/ \eqref{tgv} does not hold, then
$ v(\cC_1)/n\pto 0$ and $e(\cC_1)/n\pto 0$.
  \end{romenumerate}
\end{theorem}

\begin{proof}
We apply \refP{PT1} to $\gndxy$ as discussed above.
Note that $\py_1>0$ by \eqref{nyjlimv} and the assumption
$(1-\pi_j)p_j>0$ for some $j$.
By \eqref{nyjlimv},
  \begin{equation*}
	\begin{split}
\zeta\E\Dy(\Dy-2)
&=\zeta\sumj j(j-2)\py_j	  
=\sumji j(j-2)\pi_jp_j-\sumji j(1-\pi_j)p_j
\\&
=\sumji j(j-1)\pi_jp_j-\sumji jp_j.
	\end{split}
  \end{equation*}
Hence, the condition $\E \Dy(\Dy-2)>0$ is equivalent to \eqref{tgv}.

In particular, it follows that
$v(\cC_2)=\opn$ in \ref{TGva} and $v(\cC_1)=\opn$ in \ref{TGvb}.
That also $e(\cC_2)=\opn$ in \ref{TGva} and $e(\cC_1)=\opn$ in \ref{TGvb}
follows by \refR{Ropn} applied to $\gndxy$.

It remains only to verify the formulas \eqref{tgvxi}--\eqref{tgve}.
The equation $\gdy'(\xi)=\gly\xi$ is by \eqref{gly} equivalent to
$\zeta\gdy'(\xi)=\gl \xi$, which can be written as \eqref{tgvxi} by
\eqref{nyjlimv} and a simple calculation.

By \eqref{arnv}, using \eqref{nrv}, \eqref{zetav} and \eqref{gdy},
\begin{equation*}
  \begin{split}
	v(\cCi)/n
&\pto
\zeta-\zeta\gdy(\xi)-(1-\xi)\sumji j(1-\pi_j)p_j
\\&
=
\sumj\pi_jp_j
-\sumj \bigpar{\pi_jp_j \xi^j +  j(1-\pi_j)p_j \xi}
+\xi\sumj j(1-\pi_j)p_j
\\&
=\sumj\pi_jp_j(1-\xi^j). 
  \end{split}
\end{equation*}
Similarly, by \eqref{arne},
\eqref{gly}, \eqref{zetav} and \eqref{nrv},
\begin{equation*}
  \begin{split}
e(\cCi)/n
&\pto
\tfrac12\gl(1-\xi^2)-(1-\xi)\sumji j(1-\pi_j)p_j
\\&=
(1-\xi)\sumji j\pi_jp_j	-\frac{(1-\xi)^2}{2}\gl.
  \end{split}
\qedhere
\end{equation*}
\end{proof}

In the standard case when all $\pi_d=\pi$, this leads to a simple
criterion,
which
earlier has been shown by 
\citet{SJ199} and \citet{Fount} by different methods.
(A modification of the usual branching process argument for \gndx{} in
\cite{SJ199} and a method similar to ours in \cite{Fount}.)

\begin{corollary}[\cite{SJ199,Fount}]
  \label{CGv}
Suppose that \refCond{C1} holds and $0<\pi<1$. Then 
there exists \whp{} a giant component in $\gndx\vpi$ if and only if
\begin{equation}\label{cgv}
  \pi>\pic\=\frac{\E D}{\E D(D-1)}.
\qed
\end{equation}
\end{corollary}

\begin{remark}
  \label{Rpic}
Note that $\pic=0$ is possible; this happens if and only if 
$\E D^2=\infty$. 
(Recall that we assume $0<\E D<\infty$, see \refCond{C1}.)
Further, $\pic\ge1$ is possible too: in this case there is \whp{} no
giant component in $\gndx$
(except possibly in the special case when $p_j=0$ for all $j\neq0,2$),
and consequently none in the subgraph
$\gndx\xpi$.

Note that by \eqref{cgv}, $\pic\in(0,1)$ if and only if
$\E D <\E D(D-1) <\infty$, \ie, if and only if $0<\E D(D-2)<\infty$.
\end{remark}

\begin{remark}
  Another case treated in \cite{SJ199} (there called \textsf{E1}) is
  $\pi_d=\ga^d$ for some $\ga\in(0,1)$. \refT{TGv} gives a new proof
  that then there is a giant component if and only if $\sumji
  j(j-1)\ga^jp_j>\gl$, which also can be written
  $\ga^2g_D''(\ga)>\gl=g_D'(1)$.
(The cases \textsf{E2} and \textsf{A} in \cite{SJ199} are more
  complicated and do not follow from the results in the present paper.)
\end{remark}

For edge percolation we similarly have the following; this too has
been shown by \citet{SJ199} and \citet{Fount}. Note that the
permutation threshold $\pi$ is the same for site and bond percolation,
as observed by \citet{Fount}.

\begin{theorem}[\cite{SJ199,Fount}]
  \label{TGe}
Consider the bond percolation model $\gndx\epi$, and
suppose that \refCond{C1} holds
and that $0<\pi<1$. 
Then there is \whp{} a giant component if and only if 
\begin{equation}\label{tge}
  \pi>\pic\=\frac{\E D}{\E D(D-1)}.
\end{equation}
\begin{romenumerate}
\item\label{TGea}
If \eqref{tge} holds, 
then there is a unique $\xi=\xi\xe(\pi)\in(0,1)$
such that 
\begin{equation}\label{tgexi}
\piqq g_D'\bigpar{1-\piqq+\piqq\xi}+(1-\piqq)\gl=\gl\xi,
\end{equation}
and then
\begin{align}\label{tgev}
  v(\cC_1)/n&\pto 
\vz\xe(\pi)\=1-g_D\bigpar{1-\piqq+\piqq\xi} >0,
\\
  e(\cC_1)/n&\pto 
\ez\xe(\pi)
\=
\piqq(1-\xi)\gl-\tfrac12\gl(1-\xi)^2.
\label{tgee}
\end{align}
Furthermore,
$ v(\cC_2)/n\pto 0$ and  $ e(\cC_2)/n\pto 0$.
\item\label{TGeb}
If\/ \eqref{tge} does not hold, then
$ v(\cC_1)/n\pto 0$ and $e(\cC_1)/n\pto 0$.
  \end{romenumerate}
\end{theorem}

\begin{proof}
We argue as in the proof of \refT{TGv}, noting that $\py_1>0$  by
\eqref{nyjlime}.
By \eqref{gdye},
  \begin{equation*}
	\begin{split}
\zeta\E\Dy(\Dy-2)
&=\zeta\gdy''(1)-\zeta\gdy'(1)
=\pi g_D''(1)-\piqq g_D'(1)-(1-\piqq)\gl
\\&
=\pi\E D(D-1)-\gl,
	\end{split}
  \end{equation*}
which yields the criterion \eqref{tge}.
Further, if \eqref{tge} holds, then 
the equation $\gdy'(\xi)=\gly\xi$, which by \eqref{gly} is
equivalent to $\zeta\gdy'(\xi)=\zeta\gly\xi=\gl\xi$,
becomes \eqref{tgexi} by \eqref{gdye}.

By \eqref{arnv}, \eqref{zetae}, \eqref{nre} and \eqref{gdye},
\begin{equation*}
  \begin{split}
	v(\cCi)/n
&\pto
\zeta-\zeta\gdy(\xi)-(1-\xi)(1-\piqq)\gl
=1-g_D\bigpar{1-\piqq+\piqq\xi)},
  \end{split}
\end{equation*}
which is \eqref{tgev}.
Similarly, 
\eqref{arne}, \eqref{zetae}, \eqref{gly} and \eqref{nre} yield
\begin{equation*}
    e(\cC_1)/n\pto 
\tfrac12\gl(1-\xi^2)-(1-\xi)(1-\piqq)\gl
=\piqq(1-\xi)\gl-\tfrac12\gl(1-\xi)^2,
\end{equation*}
which is \eqref{tgee}.
The rest is as above.
\end{proof}

\begin{remark}\label{Rcritical}
It may come as a surprise that we have the same criterion \eqref{cgv}
and \eqref{tge} for site and bond percolation, since the proofs above
arrive at this equation in somewhat different ways. However, remember
that all results here are consistent with the standard branching
process approximation in \refR{Rbp}
(even if our proofs use different arguments)
and it is obvious that both site and bond percolation affect the mean
number of offspring in the branching process in the same way, namely
by multiplication by $\pi$. Cf.\ \cite{SJ199}, where the proofs are
based on such branching process approximations.
\end{remark}

Define
\begin{align}
  \label{rho}
\rho\xv=\rho\xv(\pix)\=1-\xi\xv(\pix) 
&&
\text{and}
&&&
\rho\xe=\rho\xe(\pi)\=1-\xi\xe(\pi);
\end{align}
recall from \refR{Rbp} 
that $\xi\xv$ and $\xi\xe$ are the extinction probabilities in the two
branching processes defined by the site and bond percolation models,
and thus $\rho\xv$ and $\rho\xe$ are the corresponding survival probabilities.
For bond percolation, \eqref{tgexi}--\eqref{tgee} can be written in the
somewhat simpler forms
\begin{align}
\piqq g_D'\bigpar{1-\piqq\rhoe}&=\gl(\piqq-\rhoe),
\label{tgerho}
\\ 
  v(\cC_1)/n
\pto 
\vz\xe(\pi)&\=1-g_D\bigpar{1-\piqq\rhoe(\pi)},
\label{tgerhov}
\\
  e(\cC_1)/n\pto 
\ez\xe(\pi)&\=
\piqq\gl\rhoe(\pi)-\tfrac12\gl\rhoe(\pi)^2.
\label{tgerhoe}
\end{align}
Note further that if we consider 
site percolation with all $\pi_j=\pi$, 
\eqref{tgvxi} can be written 
\begin{equation}\label{tgvrho}
  \pi \bigpar{\gl-g_D'(1-\rhov)}=\gl\rhov
\end{equation}
and it follows by comparison with \eqref{tgerho} that
\begin{equation}\label{tgwrho}
  \rhov(\pi)=\piqq\rhoe(\pi).
\end{equation}
Furthermore, \eqref{tgvv}, \eqref{tgve}, \eqref{tgerhov} and
\eqref{tgerhoe} now yield
\begin{align}\label{tgwv}
\vz\xv(\pi)
&=\pi \bigpar{1-g_D(\xi\xv(\pi))} 
=\pi \bigpar{1-g_D(1-\rhov(\pi))} 
=\pi\vz\xe(\pi),
\\
\ez\xv(\pi)&=\pi\gl\rhov(\pi)-\tfrac12\gl \rhov(\pi)^2
=\pi\ez\xe(\pi).
\label{tgwe}
\end{align}

We next consider how the various
parameters above depend on $\pi$, for both
site percolation and bond percolation, where for
site percolation we in the remainder of this section
consider only the case when all $\pi_j=\pi$.

We have so far defined the parameters for $\pi\in(\pic,1)$ only; we
extend the definitions by letting $\xi\xv\=\xi\xe\=1$ and
$\rhov\=\rhoe\=\vz\xv\=\vz\xe\=\ez\xv\=\ez\xe\=0$ for $\pi\le\pic$,
noting that this is compatible with the branching process
interpretation of $\xi$ and $\rho$ in \refR{Rbp} and that the
equalities in \eqref{tgvxi}--\eqref{tgwe} hold trivially.

\begin{theorem}
  \label{Trho} 
Assume \refCond{C1}.
The functions $\xi\xv\x \rho\xv\x \vz\xv\x \ez\xv\x
\xi\xe\x \rho\xe\x \vz\xe\x \ez\xe$
are continuous functions of $\pi\in(0,1)$ and are analytic 
except at $\pi=\pic$. 
(Hence, the functions are analytic in $(0,1)$ if and only if $\pic=0$
or $\pic\ge1$.)
\end{theorem}

\begin{proof} 
It suffices to show this for $\xi\xv$; the result for the other
functions then follows by \eqref{rho} and
\eqref{tgwrho}--\eqref{tgwe}.
Since the case $\pi\le\pic$ is trivial, it suffices to consider
$\pi\ge\pic$, and we may thus assume that $0\le\pic<1$.

If $\pi\in(\pic,1)$, then, as shown above,
$\gdy'(\xiv)=\gly\xiv$, or, equivalently,
$G(\xiv,\pi)=0$, where
$G(\xi,\pi)\=\gdy'(\xi)/\xi-\gly$ is an analytic function of
$(\xi,\pi)\in(0,1)^2$. 
Moreover, $G(\xi,\pi)$ is a strictly convex function of $\xi\in(0,1]$
for any $\pi\in(0,1)$, and $G(\xiv,\pi)=G(1,\pi)=0$; hence $\partial
G(\xiv,\pi)/\partial\xi<0$. 
The implicit function theorem now shows that $\xiv(\pi)$ is analytic
for $\pi\in(\pic,1)$.

For continuity at $\pic$, suppose $\pic\in(0,1)$ and let
$\hxi=\lim_{\ntoo}\xiv(\pi_n)$ for some sequence $\pi_n\to\pic$.
Then, writing $\Dy(\pi)$ and $\gly(\pi)$ to show the dependence on
$\pi$,
$g'_{\Dy(\pi_n)}(\xiv(\pi_n))=\gly(\pi_n)\xiv(\pi_n)$
and thus by continuity, \eg{} using \eqref{gdye},
$g'_{\Dy(\pic)}(\hxi)=\gly(\pic)\hxi$.
However, for $\pi\le\pic$, we have $\E\Dy(\Dy-2)\le0$ and then $\xi=1$
is the only solution in $(0,1]$ of $\gdy'(\xi)=\gly\xi$; hence $\hxi=1$.
This shows that $\xiv(\pi)\to1$ as $\pi\to\pic$, \ie, $\xiv$ is
continuous at $\pic$.
\end{proof}

\begin{remark}
Alternatively, 
the continuity of $\xiv$ in $(0,1)$ follows by
\refR{Rbp} and continuity of the extinction
  probability as the offspring distribution varies, 
\cf{} \cite[Lemma 4.1]{SJ199}. 
Furthermore, by the same arguments, the parameters are continuous also 
at $\pi=0$ and, except in the case when
$p_0+p_2=1$ (and thus $\Dy=1$ a.s.), at $\pi=1$ too.
\end{remark}

At the threshold $\pic$, we have linear growth of the size of the
giant component for (slightly) larger $\pi$, provided $\E D^3<\infty$,
and thus a jump discontinuity in the derivative of
$\xiv\x\vz\xv\x\dots$\,.
More precisely, the following holds.
We are here only interested in the case $0<\pic<1$, which is
equivalent to $0<\E D(D-2)<\infty$, see \refR{Rpic}.

\begin{theorem}
  \label{Tcrit}
Suppose that $0<\E D(D-2)<\infty$; thus $0<\pic<1$.
If further $\E D^3<\infty$, then as $\eps\downto0$,
\begin{align}
  \rhov(\pic+\eps)
&\sim
\frac{2\E D(D-1)}{\pic\E D(D-1)(D-2)} \eps
=
\frac{2\bigpar{\E D(D-1)}^2}{\E D\cdot\E D(D-1)(D-2)} \eps
\label{tca}
\\
  \vz\xv(\pic+\eps)
&\sim
  \ez\xv(\pic+\eps)
\sim
\pic\gl  \rhov(\pic+\eps)
\sim
\frac{2\E D\cdot \E D(D-1)}{\E D(D-1)(D-2)} \eps.
\label{tcb}
\end{align}
Similar results for $\rhoe$, $\vz\xe$, $\ez\xe$ follow by
\eqref{tgwrho}--\eqref{tgwe}. 
\end{theorem}

\begin{proof}
For $\pi=\pic+\eps\downto\pic$, by 
$g_D''(1)=\E D(D-1)=\gl/\pic$, see \eqref{cgv}, and \eqref{tgvrho},
\begin{equation}\label{cam}
  \eps g_D''(1)\rhov =(\pi-\pic)g_D''(1)\rhov
=\pi g_D''(1)\rhov-\gl\rhov = \pi\bigpar{g_D''(1)\rhov-\gl+g_D'(1-\rhov)}.
\end{equation}
Since $\E D^3<\infty$, $g_D$ is three times continuously
differentiable on $\oi$, and a Taylor expansion yields
$
g_D'(1-\rhov)
=
\gl-\rhov g_D''(1)+\rhov^2 g_D'''(1)/2+o(\rhov^2)
$.
Hence, \eqref{cam} yields, since $\rhov>0$,
\begin{equation*}
  \eps g_D''(1)
= \pi \rhov g_D'''(1)/2+o(\rhov)
= \pic \rhov g_D'''(1)/2+o(\rhov).
\end{equation*}
Thus,
noting that $g_D''(1)=\E D(D-1)$ and $g_D'''(1)=\E D (D-1)(D-2)>0$
(since $\E D(D-2)>0$),
\begin{equation*}
\rhov
\sim
\frac{2 g_D''(1)}{\pic g_D'''(1)}\eps
=\frac{2 \E D(D-1)}{\pic \E D(D-1)(D-2)}\eps,
\end{equation*}
which yields \eqref{tca}.
Finally, \eqref{tcb} follows easily by \eqref{tgwv} and \eqref{tgwe}.
\end{proof}

If $\E D^3=\infty$, we find in the same way a slower growth of
$\rhov(\pi)$, $\vz\xv(\pi)$, $\ez\xv(\pi)$ at $\pic$.
As an example, we consider $D$ with a power law tail, 
$p_k\sim ck^{-\gamma}$, 
where we take $3<\gamma<4$ so that $\E D^2<\infty$ but
$\E D^3=\infty$.

\begin{theorem}
  \label{Tcritpow}
Suppose that $p_k\sim ck^{-\gamma}$ as $k\to\infty$, where
$3<\gamma<4$ and $c>0$. Assume further that $\E D(D-2)>0$. Then
$\pic\in(0,1)$ and, as $\eps\downto0$,
\begin{align*}
  \rhov(\pic+\eps)
&\sim
\lrpar{\frac{\E D(D-1)}{c\pic\Gamma(2-\gamma)}}^{1/(\gamma-3)}
 \eps^{1/(\gamma-3)},
\\
  \vz\xv(\pic+\eps)
&\sim
  \ez\xv(\pic+\eps)
\sim
\pic\gl  \rhov(\pic+\eps)
\\&
\sim 
\pic\gl
\lrpar{\frac{\E D(D-1)}{c\pic\Gamma(2-\gamma)}}^{1/(\gamma-3)}
 \eps^{1/(\gamma-3)}.
\end{align*}
Similar results for $\rhoe$, $\vz\xe$, $\ez\xe$ follow by
\eqref{tgwrho}--\eqref{tgwe}. 
\end{theorem}

\begin{proof}
  We have, for example by comparison with the Taylor expansion of
  $\bigpar{1-(1-t)}^{\gamma-4}$,
  \begin{equation*}
g_D'''(1-t)=\sum_{k=3}^\infty k(k-1)(k-2)p_k(1-t)^k
\sim c\Gamma(4-\gamma)t^{\gamma-4},
\qquad t\downto0,	
  \end{equation*}
and thus by integration
  \begin{equation*}
g_D''(1)-g_D''(1-t)
\sim c\Gamma(4-\gamma)(\gamma-3)\qw t^{\gamma-3}
=
c|\Gamma(3-\gamma)|t^{\gamma-3},
  \end{equation*}
and, integrating once more,
  \begin{equation*}
\rhov g_D''(1)-(\gl-g_D'(1-\rhov))
\sim c\Gamma(2-\gamma) \rhov^{\gamma-2}.
  \end{equation*}
Hence, \eqref{cam} yields
  \begin{equation*}
\eps g_D''(1)\rhov
\sim c\pic\Gamma(2-\gamma) \rhov^{\gamma-2}
  \end{equation*}
and the results follow, again using \eqref{tgwv} and \eqref{tgwe}.
\end{proof}

\section{$k$-core}\label{Score}
Let $k \ge 2$ be a fixed integer. The
$k$-core of a graph $G$,
denoted by $\crk(G)$, 
is the largest induced
subgraph of $G$ with minimum vertex degree at least $k$. (Note that
the $k$-core may be empty.)
The question whether a non-empty $k$-core exists in a random graph has
attracted a lot of attention for various models of random graphs since
the pioneering papers by \citet{Bollobas84},
\citet{Luczak91} and 
\citet{PSW} for $\gnp$
and \gnm;
in particular, the case of $\gnd$ and $\gndx$ with given degree
sequences have been studied by several authors, see \citet{SJ184,SJ196} and
the references given there.

We study the percolated $\gndx\xpi$ by the exposion method presented
in \refS{Sintro}. For the $k$-core, the cleaning up stage is trivial:
by definition, the $k$-core of $\gndxy$ does not contain any vertices
of degree 1, so it is unaffected by the removal of all red vertices,
and thus 
\begin{equation}
  \label{crk=}
\crk\bigpar{\gndx\xpi}=\crk\bigpar{\gndxy}.
\end{equation}

Let, for $0\le p\le1$, $D_p$ be the thinning of $D$ obtained by
taking $D$ points and then randomly and independently keeping each of them with
probability $p$. 
Thus, given $D=d$, $D_p\sim\Bi(d,p)$.
Define, recalling the notation \eqref{bi},
\begin{align}
  h(p)&\= 
\E\bigpar{D_p\ett{D_p\ge k}}
=
\sum_{j=k}^\infty \sum_{l=j}^\infty j p_l \bljp,
\label{h}
\\
  h_1(p)&\= \P(D_p\ge k)
=\sum_{j=k}^\infty \sum_{l=j}^\infty  p_l \bljp.
\label{h1}
\end{align}
Note that 
$D_p$ is stochastically increasing in $p$, and thus
both $h$ and $h_1$ are increasing in $p$, with
$h(0)=h_1(0)=0$. Note further that $h(1)=\sum_{j=k}^\infty j p_j \le\gl$
and $h_1(1)=\sum_{j=k}^\infty  p_j \le1$, with strict inequalities
unless $p_j=0$ for all $j=1,\dots,k-1$ or $j=0,1,\dots,k-1$, respectively.
Moreover, 
\begin{equation}\label{h2}
  \begin{split}
  h(p)
&=\E D_p-\E\bigpar{D_p\ett{D_p\le k-1}}
=\E D_p-\sum_{j=1}^{k-1}j\P\xpar{D_p=j}
\\&
=\gl p-\sum_{j=1}^{k-1}\sum_{l\ge j}j p_l\binom lj p^ j(1-p)^{l-j}
\\&
=\gl p-\sum_{j=1}^{k-1}\frac{p^j}{(j-1)!}g_D^{(j)}(1-p).
  \end{split}
\end{equation}
Since $g_D(z)$ is an analytic function in \set{z:|z|<1}, \eqref{h2}
shows that $h(p)$ 
is an analytic function in the domain \set{p:|p-1|<1} in the complex
plane; in particular, $h$ is analytic on $\ooi$. 
(But not necessarily at 0, as seen by \refE{Einfty}.)
Similarly, $h_1$ is analytic on $\ooi$.

We will use the following result by
\citet[Theorem 2.3]{SJ184}. 
\begin{proposition}[\cite{SJ184}]
  \label{PTK}
Suppose that  \refCond{C1} holds.
Let $k\ge2$ be fixed, and let $\crkx$ be the \kcore{} of \gndx.
Let $\pmax\=\max\set{p\in\oi:h(p)=\gl p^2}$. 
  \begin{romenumerate}
\item
If $h(p)<\gl p^2$ for all $p\in(0,1]$, 
which is equivalent to $\pmax=0$,
then 
$\crkx$ has $\opn$ vertices and $\opn$  
edges. Furthermore, 
if also $k\ge3$ and $\sumin e^{\ga d_i} =O(n)$ for
some $\ga>0$, then \crkx{} is
empty \whp.
\item
If $h(p)\ge \gl p^2$ for some $p\in(0,1]$, 
which is equivalent to $\pmax\in(0,1]$,
and further $\pmax$ is not a local maximum point of $h(p)-\gl p^2$, 
then
\begin{align}
  v(\crkx)/n&\pto h_1(\pmax)>0,
\label{ptkv}
\\
  v_j(\crkx)/n&\pto \P(\Dpmax=j)=\sumlj p_l\bxxx lj{\pmax},
\qquad j\ge k,
\label{ptkvj}
\\
e(\crkx)/n&\pto h(\pmax)/2=\gl \pmax^2/2.
\label{ptke}
\end{align}
  \end{romenumerate}
\end{proposition}

\begin{remark}
The result \eqref{ptkvj} is not stated explicitly in \cite{SJ184}, but
as remarked in \cite[Remark 1.8]{SJ196}, it follows immediately from the
proof in \cite{SJ184} of \eqref{ptkv}. (Cf.\ \cite{CW} for the random
graph $G(n,m)$.) 
\end{remark}

\begin{remark}
  The extra condition in (ii)
that $\pmax$ is not a local maximum point of
 $h(p)-\gl p^ 2$ is actually stated somewhat differently in
 \cite{SJ184}, \viz{} as $\gl p^2<h(p)$ in some interval
 $(\pmax-\eps,\pmax)$.
However, since $g(p)\=h(p)-\gl p^2$ is analytic at $\pmax$, a Taylor
expansion at $\pmax$ shows that for some such interval
 $(\pmax-\eps,\pmax)$, either $g(p)>0$ or $g(p)\le0$ throughout the
interval. Since $g(\pmax)=0$ and $g(p)<0$ for $\pmax<p\le1$, the two
versions of the condition are equivalent.

The need for this condition is perhaps more clearly seen in the
percolation setting, \cf{} \refR{Rjumps}.
\end{remark}

There is a natural interpretation of this result in terms of the
branching process approximation of the local exploration process,
similar to the one described for the giant component in \refR{Rbp}. For
the \kcore, this was observed already by \citet{PSW}, but (unlike for the
giant component), the branching process approximation has so far
mainly been used heuristically; the technical difficulties to make a
rigorous proof based on it are formidable, and have so far been
overcome only by \citet{Riordan:kcore} for a related random graph
model. We, as most others, avoid this complicated method of proof, and
only identify the limits in \refP{PTK} (which is proved by other, simpler,
methods in \cite{SJ184}) with quantities for the branching process.
Although this idea is not new, we have, for the random graphs that we
consider,
not seen a detailed proof of it in the literature, so
for completeness we provide one in \refApp{App}.

\begin{remark}\label{Rk=2}
If $k=2$, then \eqref{h2} yields
\begin{equation}\label{hk=2}
  \begin{split}
  h(p)
&
=\gl p-\suml p_llp(1-p)^{l-1}
=\gl p-pg_D'(1-p)	
  \end{split}
\end{equation}
and thus
\begin{equation*}
h(p)-\gl p^2=p\bigpar{\gl(1-p)-g_D'(1-p)}.
\end{equation*}
It follows that $\pmax=1-\xi$, where $\xi$ is as in \refP{PT1} and
\refR{RPT1ii}; \ie, by \refR{Rbp}, 
$\pmax=\rho$, the survival probability of the
branching process $\cX$ with offspring distribution $\D$. 
(See \refApp{App} for further explanations of this.)
\end{remark}

We now easily derive results for the \kcore{} in the percolation models.
For simplicity, we consider for site percolation only the case when
all $\pi_k$ are equal; the general case is similar but the explicit
formulas are less nice.

\begin{theorem}
    \label{TKv}
Consider the site percolation model $\gndx\vpi$
with $0\le\pi\le1$, and
suppose that \refCond{C1} holds.
Let $k\ge2$ be fixed, and let 
$\crkx$ be the \kcore{} of $\gndx\vpi$.
Let
\begin{equation}\label{pick}
  \pic=\pick
\=
\inf_{0<p\le1}\frac{\gl p^2}{h(p)} 
= \lrpar{\sup_{0<p\le1}\frac{h(p)}{\gl p^2}}\qw.
\end{equation}
  \begin{romenumerate}
\item
If $\pi<\pic$, then
$\crkx$ has $\opn$ vertices and $\opn$  
edges. 
Furthermore, 
if also $k\ge3$ and  $\sumin e^{\ga d_i} =O(n)$ for
some $\ga>0$, then \crkx{} is
empty \whp.
\item
If $\pi>\pic$, then \whp{} \crkx{} is non-empty.
Furthermore, if $\pmax=\pmax(\pi)$ is the largest $p\le1$ such that
$h(p)/(\gl p^2)=\pi\qw$, 
and $\pmax$ is not a local maximum point of $h(p)/(\gl p^2)$ in $\ooi$, 
then
\begin{align*}
v(\crkx)/n&\pto \pi h_1(\pmax)>0,
\\
v_j(\crkx)/n&\pto \pi \P(\Dpmax=j),\qquad j\ge k,
\\
e(\crkx)/n&\pto \pi h(\pmax)/2=\gl \pmax^2/2.
\end{align*}
\end{romenumerate}
\end{theorem}

\begin{theorem}
    \label{TKe}
Consider the bond percolation model $\gndx\epi$ with $0\le\pi\le1$, and
suppose that \refCond{C1} holds.
Let $k\ge2$ be fixed, and let 
$\crkx$ be the \kcore{} of $\gndx\epi$.
Let $  \pic=\pick$ be given by \eqref{pick}.
  \begin{romenumerate}
\item
If $\pi<\pic$, then
$\crkx$ has $\opn$ vertices and $\opn$  
edges.
Furthermore, 
if also $k\ge3$ and  $\sumin e^{\ga d_i} =O(n)$ for
some $\ga>0$, then \crkx{} is
empty \whp.
\item
If $\pi>\pic$, then \whp{} \crkx{} is non-empty.
Furthermore, if $\pmax=\pmax(\pi)$ is the largest $p\le1$ such that
$h(p)/(\gl p^2)=\pi\qw$, 
and $\pmax$ is not a local maximum point of $h(p)/(\gl p^2)$ in $\ooi$, 
then
\begin{align*}
v(\crkx)/n&\pto  h_1(\pmax)>0,
\\
v_j(\crkx)/n&\pto  \P(\Dpmax=j),\qquad j\ge k,
\\
e(\crkx)/n&\pto  h(\pmax)/2=\gl \pmax^2/(2\pi).
\end{align*}
\end{romenumerate}
\end{theorem}

For convenience, we define $\gf(p)\=h(p)/p^2$, $0<p\le1$.

\begin{remark}\label{RTK}
  Since $h(p)$ is analytic in $(0,1)$, there is at most a countable
number of local maximum points of $\gf(p)\=h(p)/p^2$ (except when $h(p)/p^2$
is constant), 
and thus at most a countable number of local maximum values of
$h(p)/p^2$. Hence, there is at most a countable number of exceptional
values of $\pi$ in part (ii) of Theorems \refand{TKv}{TKe}. 
At these exceptional values, we have a discontinuity of 
$\pmax(\pi)$ and thus of 
the relative asymptotic
size $\pi h_1(\pmax(\pi))$  or $h_1(\pmax(\pi))$ 
of the \kcore; in other words, there is a
phase transition of the \kcore{} at each such exceptional $\pi$.
(See Figures \refand{FPoisson}{F10**i}.)
Similarly, if $\gf$ has an inflection point at $\pmax(\pi)$, \ie, if 
$\gf'(\pmax)=\gf''(\pmax)=\dots=\gf^{(2\ell)}(\pmax)=0$ and 
$\gf^{(2\ell+1)}(\pmax)<0$ for some $\ell\ge1$, 
then $\pmax(\pi)$ and  $h_1(\pmax(\pi))$ are continuous but the
derivatives of $\pmax(\pi)$ and 
$h_1(\pmax(\pi))$ become infinite at this point, so we have a phase
transition of a different type. For all other $\pi>\pic$, the implicit
function theorem shows that $\pmax(\pi)$ and 
$h_1(\pmax(\pi))$ are analytic at $\pi$.

Say that $\tp$ is a \emph{\crp} of $\gf$ if $\gf'(\tp)=0$, and a
\emph{bad \crp} if further, $\tp\in(0,1)$, $\gf(\tp)>\gl$ and
$\gf(\tp)>\gf(p)$ for all $p\in(\tp,1)$. 
It follows that there is a
1--1 correspondence between \phtr{s} in $(\pic,1)$ or $[\pic,1)$
and bad \crp{s} $\tp$ of $\gf$, with
the \phtr{} occurring at $\tpi=\gl/\gf(\tp)$.
This includes $\pic$ if and only if $\sup_{\ooi}\gf(p)$ is attained
and larger than $\gl$,
in which case the last global maximum point is a bad \crp; if this
supremum is finite not attained, then there is another first-order
\phtr{} at $\pic$, while if the supremum is infinite, then $\pic=0$.
Finally, there may be a further \phtr{} at $\tpi=1$ (with $\tp=1$);
this happens if and only if $\gf(1)=h(1)=\gl$ and $\gf'(1)\le0$.

The \phtr{s} are first-order when the corresponding $\tp$  is a bad
\lmp{} of $\gf$, \ie, a bad \crp{} that is a \lmp. (This includes
$\pic$ when $\sup_{\ooi}\gf$ is attained, but not otherwise.) 
Thus, the \phtr{} that occur are typically first order, but there are
exceptions, see Examples \refand{Einfty}{Esecond}.
\end{remark}

\begin{remark}\label{Rjumps}
  The behaviour at $\pi=\pic$ depends on more detailed properties of
  the degree sequences \dnn, or equivalently of $\xD_n$.
Indeed, more precise results can be derived from
\citet[Theorem 3.5]{SJ196}, at least under somewhat stricter
  conditions on \dnn; in particular, it then follows that the width of the
threshold is of the order $n\qqw$, \ie, that there is a sequence
  $\picn$ depending on \dnn, with $\picn\to\pic$, such that
  $\gndx\vpi$ and   $\gndx\epi$ 
\whp{} have a non-empty \kcore{} if
  $\pi=\picn+\go(n)n\qqw$
with $\go(n)\to\infty$, but \whp{} an empty \kcore{} if
  $\pi=\picn-\go(n)n\qqw$, while in the intermediate case
  $\pi=\picn+cn\qqw$ with $-\infty<c<\infty$,
$\P(\gndx\xpi\text{ has a non-empty \kcore})$ converges to a limit
  (depending on $c$) in $(0,1)$. We leave the details to the reader.

The same applies to further phase transitions that may occur.
\end{remark}

\begin{remark}\label{Rk=2b}
If $k=2$, then \eqref{hk=2} yields
\begin{equation*}
 \gf(p)
\=h(p)/p^2
=\sum_{j\ge2} p_jj(1-(1-p)^{j-1})/p,
\end{equation*}
which is decreasing on $\ooi$ (or constant, when
$\P(D>2)=0$), with 
\begin{equation*}
\sup_{p\in\ooi}\gf(p)=\lim_{p\to0}\gf(p)=\sum_jp_jj(j-1)
=\E D(D-1)\le\infty.
\end{equation*}
Hence 
$$\picii=\gl/\E D(D-1)=\E D/\E D(D-1),$$ 
coinciding with the
critical value in \eqref{cgv} for a giant component.

Although there is no strict implication in any direction between ``a
giant component'' and ``a non-empty 2-core'', in random graphs these
seem to typically appear 
together (in the form of a large connected component of the 2-core),
see \refApp{App} for branching process heuristics explaing this.
\end{remark}

\begin{remark}
  We see again that the results for site and bond percolation are
  almost identical. 
In fact, they become the same if we measure the size of the \kcore{}
  in relation to the size of the percolated graph $\gndx\xpi$, since
  $v(\gndx\epi)=n$ but  
 $v(\gndx\vpi)\sim\Bi(n,\pi)$, so  
$v(\gndx\vpi)/n\pto\pi$. 
Again, this is heuristically explained by the
  branching process approximations; see \refApp{App} and note that 
random deletions of vertices or edges yield the
  same result in the branching process, assuming that we do not delete
  the root.
\end{remark}

\begin{proof}[Proof of \refT{TKv}]
The case $\P(D\ge k)=0$ is trivial; in this case $h(p)=0$ for all $p$
and $\pic=0$ so (i) applies. Further, \refP{PTK}(i) applies to \gndx,
and the result follows from 
$\crk(\gndx\vpi)\subseteq\crk(\gndx)$.
In the sequel we thus assume $\P(D\ge k)>0$, which implies $h(p)>0$
and $h_1(p)>0$
for $0<p\le1$.

We apply \refP{PTK} to the exploded graph $\gndxy$, recalling \eqref{crk=}.
For site percolation, $\gndx\vpi$, 
$\py_j= \zeta\qw\pi p_j$ for $j\ge 2$ 
by \eqref{nyjlimv}, and thus 
\begin{equation*}
  \P(\Dy_p=j)=\zeta\qw\pi \P(D_p=j), 
\qquad j\ge2, 
\end{equation*}
and, because $k\ge2$,
\begin{align}
\hy(p)&\=\E\bigpar{\Dy_p\ett{\Dy_p\ge k}}
  =\zeta\qw\pi h(p),
\label{chu0}
\\
\hy_1(p)&\=
\P(\Dy_p\ge k)
=
\zeta\qw\pi h_1(p).
\label{chu2}
\end{align}
Hence, the condition $\hy(p)\ge\gly p^2$ can, using \eqref{gly}, be
written
\begin{equation}\label{chu1}
 \pi h(p)\ge \gl p^2. 
\end{equation}

If $\pi<\pic$, then for every $p\in\ooi$, by \eqref{pick},
$\pi<\pic\le\gl p^2/h(p)$ so \eqref{chu1} does not hold and
$\hy(p)<\gly p^2$. Hence \refP{PTK}(i) applies to $\gndxy$, which
proves (i); note that if 
$\sumin e^{\ga d_i} =O(n)$ for some $\ga>0$, then also
\begin{equation*}
  \sum_{i=1}^{\ny} e^{\ga \y d_i}
\le
  \sum_{i=1}^{n} e^{\ga d_i}+\nr e^\ga
\le
  \sum_{i=1}^{n} e^{\ga d_i}+e^\ga\sum_{j\ge1} jn_j
=O(n).
\end{equation*}

If $\pi>\pic$, then there exists $p\in\ooi$ such that $\pi>\gl
p^2/h(p)$ and thus \eqref{chu1} holds and \refP{PTK}(ii) applies to
\gndxy. Moreover, $\pmax$ in \refP{PTK}(ii) is the largest $p\le1$
such that \eqref{chu1} holds. Since $\pi h(1)\le h(1)\le\gl$ and $h$
is continuous, we have equality in \eqref{chu1} for $p=\pmax$, \ie,
$\pi h(\pmax)=\gl\pmax^2$, so $\pmax$ is as asserted the largest
$p\le1$ with $h(p)/(\gl p^2)=\pi\qw$.

Further, if $\pmax$ is a local maximum point of $\hy(p)-\gly
p^2=\zeta\qw(\pi h(p)-\gl p^2)$, then 
$\pi h(p)-\gl p^2\le0$ in a neighbourhood of $\pmax$ and thus
$h(p)/(\gl p^2)\le1/\pi =h(\pmax)/(\gl\pmax^2)$ there; 
thus $\pmax$ is a local maximum
point of  $h(p)/(\gl p^2)$.
Excluding such points,  we obtain from
\refP{PTK}(ii) using \eqref{crk=}, \eqref{zetav},
\eqref{nyjlimv}, \eqref{gly},
\eqref{chu0} and \eqref{chu2},
\begin{align*}
  \frac{v(\crkx)}{n}
&=
 \frac{\ny}n\cdot \frac{v(\crkx)}{\ny}
\pto\zeta\hy_1(\pmax)
=\pi h_1(\pmax),
\\
  \frac{v_j(\crkx)}{n}
&=
 \frac{\ny}n\cdot \frac{v_j(\crkx)}{\ny}
\pto\zeta\P(\Dypmax=j)
=\pi \P(\Dpmax=j),
\qquad j\ge k,
\\
  \frac{e(\crkx)}{n}
&=
 \frac{\ny}n\cdot \frac{e(\crkx)}{\ny}
\pto
\zeta\frac{\gly\pmax^2}2
=\frac{\gl\pmax^2}2.
\end{align*}

This proves the result when $\pmax$ is not a local maximum point of
$h(p)/(\gl p^2)$. In particular, since $h_1(\pmax)>0$,  $\crkx$ is
non-empty \whp{} when $\pi>\pic$ is not a local maximum value of
$h(p)/(\gl p^2)$.
Finally, even if $\pi$ is such a local maximum value, we can find $\pi'$
with $\pic<\pi'<\pi$ that  is not, because by \refR{RTK} there is only
a countable number of exceptional $\pi$. By what we just have shown,
$\gndx\vpiw$ has \whp{} a non-empty $k$-core, and thus so has
$\gndx\vpi\supseteq\gndx\vpiw$. 
\end{proof}

\begin{proof}[Proof of \refT{TKe}]
We argue as in the proof just given of \refT{TKe}, again using
\eqref{crk=} and applying \refP{PTK} to the exploded graph \gndxy.
We may again assume $\P(D\ge k)>0$, and thus $h(p)>0$
and $h_1(p)>0$
for $0<p\le1$.
We may further assume $\pi>0$.

The main difference from the site percolation case is that for bond percolation
$\gndx\epi$, \eqref{nyjlime} yields
\begin{equation*}
  \P(\Dy=j)=\zeta\qw\P(D_{\piqq}=j),
\qquad j\ge2,
\end{equation*}
and hence
\begin{equation}
\label{wind}
  \P(\Dy_p=j)=\zeta\qw\P(D_{p\piqq}=j),
\qquad j\ge2,
\end{equation}
and thus
\begin{align}
\hy(p)&\=\E\bigpar{\Dy_p\ett{\Dy_p\ge k}}
  =\zeta\qw h(p\piqq),
\\
\hy_1(p)&\=
\P(\Dy_p\ge k)
=
\zeta\qw h_1(p\piqq).
\label{win1}
\end{align}
Consequently, the condition $\hy(p)\ge\gly p^2$ can, using \eqref{gly}, be
written as $h(p\piqq)\ge\gl p^2$, or
\begin{equation}\label{winston}
 \pi h(p\piqq)\ge \gl \bigpar{p\piqq}^2. 
\end{equation}

If $\pi<\pic$, then for every $p\in\ooi$
we have $p\piqq\in\ooi$ and thus by \eqref{pick}
\begin{equation*}
\pi<\pic\le\frac{\gl (p\piqq)^2}{h(p\piqq)}  
\end{equation*}
so \eqref{winston} does not hold and
$\hy(p)<\gly p^2$. Hence \refP{PTK}(i) applies to $\gndxy$ as in the
proof of \refT{TKv}.

If $\pi>\pic$, then there exists $p\in\ooi$ such that $\pi>\gl
p^2/h(p)$ and, as before,  $\pmax$ is the largest such $p$ and
satisfies $h(\pmax)/(\gl\pmax^2)=\pi\qw$.
Furthermore, if $\piqq<p\le1$, then
\begin{equation}
\pi h(p)\le\pi h(1)\le\pi\gl<\gl p^2,
\end{equation}
and thus $p\neq\pmax$. Hence $\pmax\le\piqq$.
Let $\pmaxe\=\pmax/\piqq$. Then $\pmaxe\in\ooi$ and $\pmaxe$ is the
largest $p\le1$ such that \eqref{winston} holds; \ie, the largest
$p\le1$ such that $\hy(p)\ge\gly p^2$.
We thus can apply \refP{PTK}(ii) to \gndxy, with $\pmax$ replaced by
$\pmaxe$, noting that if $\pmaxe$ is a \lmp{} of $\hy(p)-\gly p^2$,
  then $\pmax$ is a \lmp{} of
  \begin{equation*}
\hy(p\pi\qqw)-\gly(p\pi\qqw)^2
=
\zeta\qw\bigpar{h(p)-\pi\qw\gl p^2}	
  \end{equation*}
and thus of $\pi h(p)-\gl p^2$, which as in the proof of \refT{TKv}
implies that $\pmax$ is a \lmp{} of $h(p)/(\gl p^2)$.
(The careful reader may note that there is no problem with the special
case $\pmaxe=1$, when we only consider a one-sided maximum at
$\pmaxe$: in this case $\pmax=\piqq$ and $\pi
h(\pmax)=\gl\pmax^2=\gl\pi$ so $h(\pmax)=\gl$ and $\pmax=1$, $\pi=1$.)
Consequently, when $\pmax$ is not a \lmp{} of $h(p)/(\gl p^2)$,
\refP{PTK}(ii) yields, using \eqref{zetae}, \eqref{wind}, \eqref{win1},
\begin{align*}
  \frac{v(\crkx)}{n}
&=
 \frac{\ny}n\cdot \frac{v(\crkx)}{\ny}
\pto\zeta\hy_1(\pmaxe)
= h_1(\pmax),
\\
  \frac{v_j(\crkx)}{n}
&=
 \frac{\ny}n\cdot \frac{v_j(\crkx)}{\ny}
\pto\zeta\P(\Dypmaxe=j)
= \P(\Dpmax=j),
\qquad j\ge k,
\\
  \frac{e(\crkx)}{n}
&=
 \frac{\ny}n\cdot \frac{e(\crkx)}{\ny}
\pto
\zeta\frac{\gly\pmaxe^2}2
=\frac{\gl\pmax^2}{2\pi}.
\end{align*}
The proof is completed as before.
\end{proof}

Consider now what Theorems \refand{TKv}{TKe} imply for the \kcore{} as
$\pi$ increases from 0 to 1. (We will be somewhat informal; the
statements below should be interpreted as asymptotic as \ntoo{} for
fixed $\pi$, but we for simplicity omit ``\whp'' and ``$\pto$''.)

If $k=2$, we have by \refR{Rk=2b} a similar behaviour as for the giant
component in \refS{Sgiant}: in the interesting case $0<\pic<1$, the
2-core is small, $o(n)$, for $\pi<\pic$ and large, $\Theta(n)$, for
$\pi>\pic$, with a relative size $h_1(\pmax(\pi))$ that is a
continuous function of $\pi$ also at $\pic$ and analytic everywhere
else in $(0,1)$, \cf{} \refT{Trho}.

Assume now $k\ge3$. For the random graph $G(n,p)$ with $p=c/n$, the
classical result by \citet{PSW} shows that there is a first-order
(=discontinuous) phase transition at some value $c_k$;
for $c<c_k$ the \kcore{} is empty and for $c>c_k$ it is non-empty and
with a relative size $\psi_k(c)$ that jumps to a positive value at
$c=c_k$, and thereafter is analytic.
We can see this as a percolation result, choosing a large $\gl$ and
regarding $G(n,c/n)$ as obtained by bond percolation on \gnln{} with
$\pi=c/\gl$ for $c\in[0,\gl]$; $\gnln$ is not exactly a random graph
of the type \gndx{} studied in the present paper, but as said in the
introduction, it can be treated by our methods by conditioning on the
degree sequence, and it has the asymptotic degree distribution
$D\sim\Po(\gl)$. In this case, see \refE{EPo} and \refF{FPoisson}, 
$\gf$ is unimodal, with
$\gf(0)=0$, a maximum at some interior point $p_0\in(0,1)$, and
$\gf'<0$ on $(p_0,1)$.
This is a typical case; $\gf$ has these properties for many other
degree distributions too (and $k\ge3$), and these properties of $\gf$
imply by Theorems \refand{TKv}{TKe} that there is, provided
$\gf(p_0)>\gl$, a first-order phase transition at
$\pi=\pic=\gl/\gf(p_0)$ where the \kcore{} suddenly is created with a
positive fraction $h_1(p_0)$ of all vertices, but no other phase
transitions since $h_1(\pmax(\pi))$ is analytic on $(\pic,1)$.
Equivalently, recalling \refR{RTK}, we see that $p_0$ is the only bad
\crp{} of $\gf$.

However, there are other possibilities too; 
there may be several bad \crp{s} of $\gf$, and thus several \phtr{s}
of $\gf$. There may even be an infinite number of them.
We give some
examples showing different possibilities that may occur.
(A similar example with several phase transitions for a related
hypergraph process is given by \citet{DarlingLN}.)

\begin{figure}[htbp]
  \begin{center}
\includegraphics[width=8cm,angle=-90]{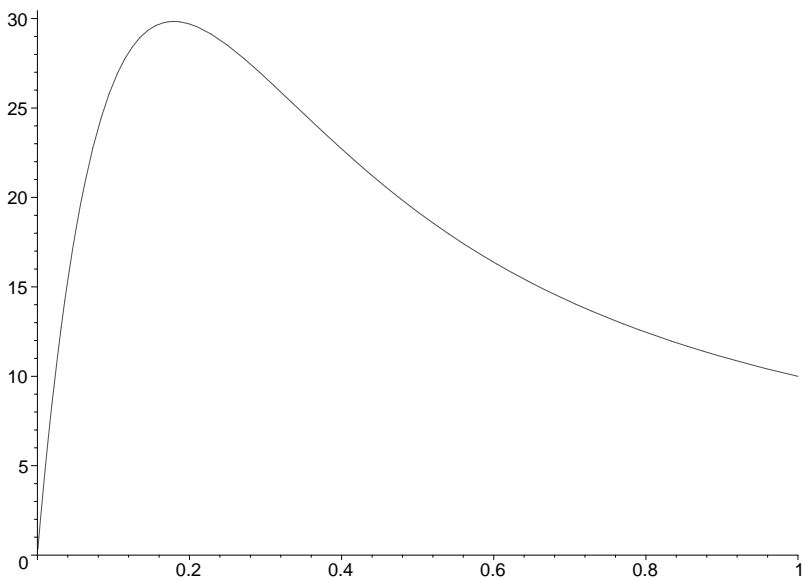}
  \end{center}
\caption{$\gf(p)=h(p)/p^2$ for $D\sim\Po(10)$ and $k=3$.}
\label{FPoisson}
\end{figure}

\begin{figure}[htbp]
  \begin{center}
\includegraphics[width=8cm,angle=-90]{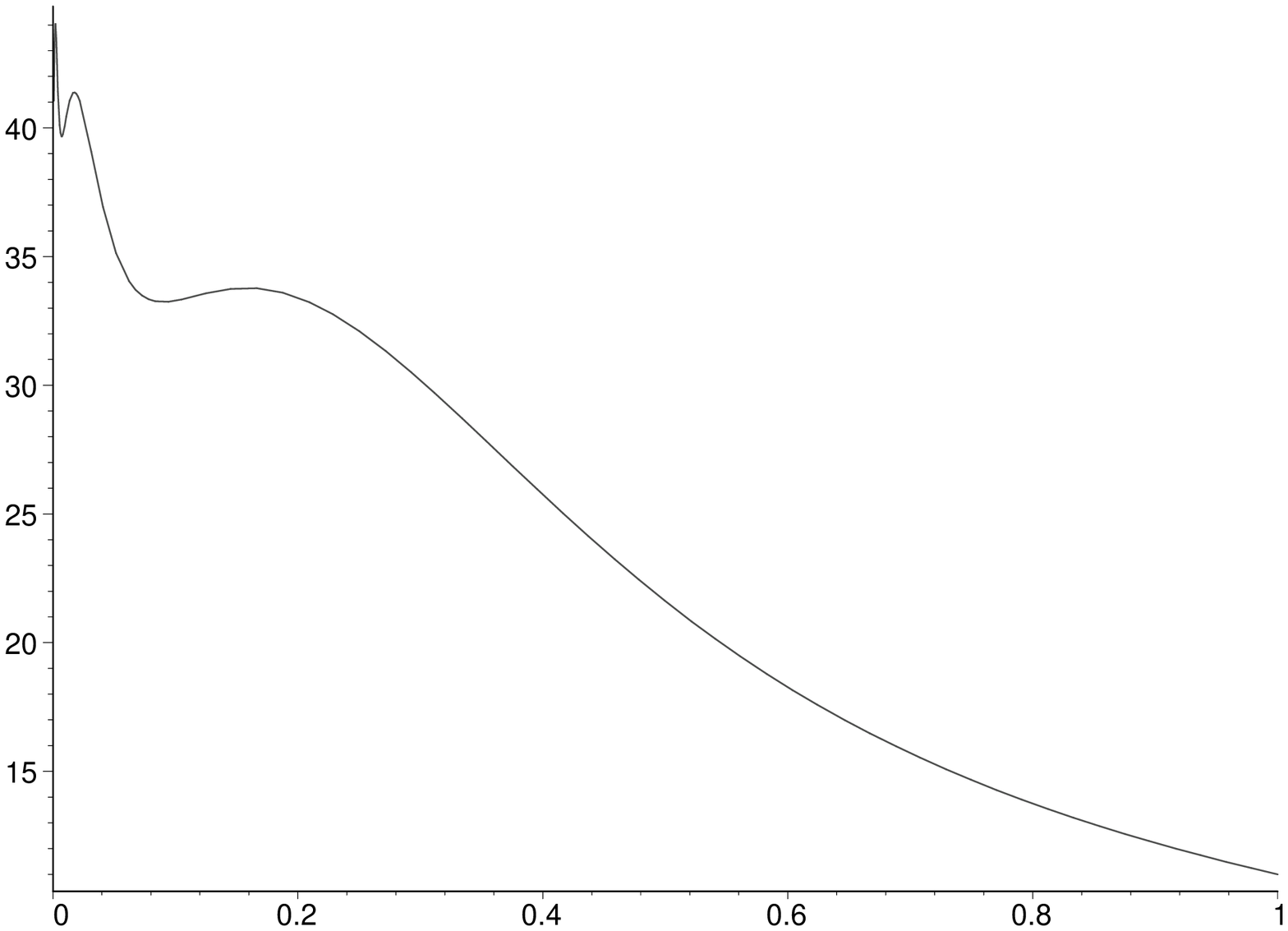}
  \end{center}
\caption{$\gf(p)=h(p)/p^2$ for $k=3$ and
$p_{10^i}=99\cdot10^{-2i}$, $i=1,2,\dots$,
  ($p_j=0$ for all other $j$). Cf.\ \refE{E2**i}.}
\label{F10**i}
\end{figure}

\begin{example}
  \label{EPo}
A standard case is when $D\sim\Po(\gl)$ and $k\ge3$. 
(This includes, as said above, the case $\gnln$  by conditioning on the degree
sequence, in which case we recover the
result by \cite{PSW}.)
Then $D_p\sim\Po(\gl p)$ and a simple calculation shows that 
$h(p)=\gl p \P(\Po(\gl p)\ge k-1)$, see \cite[p.\ 59]{SJ184}.
Hence, if $c_k\=\min_{\mu>0}\mu/\P(\Po(\mu)\ge k-1)$ and
$\gl> c_k$, then
$\pic=\inf_{0<p\le1}\bigpar{\gl p^2/h(p)}
=c_k/\gl$.
Moreover, it is easily shown that $h(p)/p^2$ is unimodal, see
\cite[Lemma 7.2]{SJ184} and \refF{FPoisson}.
Consequently, there is as discussed above a single first-order phase
transition at $\pi=c_k/\gl$ \cite{PSW}.
\end{example}

\begin{example}
  \label{E2}
Let $k=3$ and consider graphs with only two vertex degrees, 3 and $m$,
say, with $m\ge4$. Then, \cf{} \eqref{h2},
\begin{equation*}
  \begin{split}
h(p)&=3p_3\P(D_p=3\mid D=3) + p_m\E(D_p-D_p\ett{D_p\le2}\mid D=m)	
\\
&=3p_3p^3+p_m\bigpar{mp-mp(1-p)^{m-1}-m(m-1)p^2(1-p)^{m-2}}.
  \end{split}
\end{equation*}
Now, let $p_3\=1-a/m$ and $p_m\=a/m$, with $a>0$ fixed and $m\ge a$,
and let $m\to\infty$. Then, writing $h=h_m$,
$h_m(p)\to3p^3+ap$ for $p\in\ooi$ and thus
\begin{equation*}
 \gf_m(p)\=\frac{h_m(p)}{p^2}\to\gfoo(p)\=3p+\frac ap. 
\end{equation*}
Since
$\gfoo'(1)=3-a$, we see that if we choose $a=1$, say, then
$\gfoo'(1)>0$. Furthermore, then $\gfoo(1/4)=\tfrac34+4>\gfoo(1)=4$.
Since also
$
\gf_m'(1)=3p_3-mp_m=3p_3-a\to\gfoo'(1)$, it follows that if $m$ is
large enough, then $\gf_m'(1)>0$ but $\gf_m(1/4)>\gf_m(1)$. We fix
such an $m$ and note that $\gf=\gf_m$ is continuous on \oi{} with
$\gf(0)=0$, because the sum in \eqref{h} is finite with each term
$O(p^3)$.

Let $\ppp$ be the \gmp{} of $\gf$ in \oi. (If not unique, take the
largest value.)
Then, by the properties just shown, $\ppp\neq0$ and $\ppp\neq1$, so
$\ppp\in(0,1)$; moreover, 1 is a \lmp{} but not a \gmp.
Hence, $\pic=\gl/\gf(\ppp)$ is a first-order \phtr{} where the
3-core suddenly becomes non-empty and containing a positive fraction
$h_1(\ppp)$ of all (remaining) vertices.
There is another \phtr{} at $\pi=1$. We have $\gf(1)=h(1)=\gl$, but
since $\gf'(1)>0$, if $\ppl\=\sup{p<1:\gf(p)>\gf(1)}$, then
$\ppl<1$. Hence, as $\pi\upto1$, $\pmax(\pi)\upto\ppl$ and
$h(\pmax(\pi))\upto h_1(\ppl)<1$. Consequently, the size of the 3-core
jumps again at $\pi=1$. 

For an explicit example, numerical calculations (using \maple) show
that we can take
$a=1$ and $m=12$, or $a=1.9$ and $m=6$.
\end{example}

\begin{example}
  \label{Einfty}
Let $k=3$ and let $D$ be a mixture of Poisson distributions:
\begin{equation}
  \label{einfty}
\P(D=j)=p_j=\sum_i q_i\P(\Po(\gl_i)=j),
\qquad j\ge0,
\end{equation}
for some finite or infinite sequences
$(q_i)$ and $(\gl_i)$ with $q_i\ge0$, $\sum_i q_i=1$ and $\gl_i\ge0$.
In the case $D\sim\Po(\gl)$ we have, \cf{} \refE{EPo},
$D_p\sim\Po(\gl p)$ and thus
\begin{equation*}
  \begin{split}
	h(p)&=\E D_p-\P(D_p=1)-2\P(D_p=2)
=\gl p-\gl pe^{-\gl p}-(\gl p)^2e^{-\gl p}
\\&
=(\gl p)^2f(\gl p),
  \end{split}
\end{equation*}
where $f(x)\=\bigpar{1-(1+x)e^{-x}}/x$.
Consequently, by linearity, for $D$ given by \eqref{einfty},
\begin{equation}
  \begin{split}
	h(p)=\sum_iq_i(\gl_i p)^2f(\gl_i p),
  \end{split}
\end{equation}
and thus
\begin{equation}\label{gf1}
	\gf(p)=\sum_iq_i\gl_i^2f(\gl_i p).
\end{equation}

As a specific example, take $\gl_i=2^i$ and $q_i=\gl_i\qww=2^{-2i}$,
$i\ge1$, and add $q_0=1-\sum_{i\ge1}q_i$ and $\gl_0=0$ to make $\sum q_i=1$.
Then
\begin{equation}\label{magnus}
	\gf(p)=\sumii f(2^i p).
\end{equation}
Note that $f(x)=O(x)$ and $f(x)=O(x\qw)$ for 
$0<x<\infty$. Hence, the sum in \eqref{magnus} converges uniformly on
every compact interval $[\gd,1]$; moreover, if we define
\begin{equation}
  \psi(x)\=\sum_{i=-\infty}^\infty f(2^i x),
\end{equation}
then the sum converges uniformly on compact intervals of $(0,\infty)$
and 
$|\gf(p)-\psi(p)|=O(p)$.
Clearly, $\psi$ is a multiplicatively periodic function on $\ooo$:
$\psi(2x)=\psi(x)$; we compute the Fourier series of the periodic
function $\psi(2^y)$ on $\bbR$ and find
$\psi(2^y)=\sumnn\hpsi(n) e^{2\pi\ii ny}$ with, using integration by
parts,
\begin{equation*}
  \begin{split}
\hpsi(n)
&=\intoi\psi(2^y)e^{-2\pi\ii ny}\ddx y	
=\intoi\sumjj f(2^j2^y)e^{-2\pi\ii ny}\ddx y	
\\&
=\sumjj\intoi f(2^{j+y})e^{-2\pi\ii ny}\ddx y	
=\intoooo f(2^{y})e^{-2\pi\ii ny}\ddx y	
\\&
=\intoo f(x)x^{-2\pi\ii n/\ln2}\frac{\ddx x}{x\ln2}	
\\&
=\frac1{\ln2}\intoo \bigpar{1-(1+x)e^{-x}}x^{-2\pi\ii n/\ln2-2}\ddx x
\\&
=\frac1{\ln2(2\pi\ii n/\ln2+1)}\intoo xe^{-x}x^{-2\pi\ii n/\ln2-1}
\ddx x
\\&
=
\frac{\Gamma(1-2\pi\ii n/\ln2)}{\ln2+2\pi\ii n}.
  \end{split}
\end{equation*}

Since these Fourier coefficients are non-zero, we see that
$\psi$ 
is a non-constant
continuous function on $\ooo$ with multiplicative period 2.
Let $a>0$ be any point that is a global maximum of $\psi$, 
let $b\in(a,2a)$ be a point that is not,
and let
$I_j\=[2^{-j-1}b,2^{-j}b]$. Then $\psi$ attains its global maximum at
the interior point $2^{-j} a$ in $I_j$, and since $\gf(p)-\psi(p)=O(2^{-j})$
for $p\in I_j$, it follows that if 
$j$ is large enough, then $\gf(2^{-j}a)>\max(\gf(2^{-j-1}b),\gf(2^{-j}b))$.
Hence, if the
maximum of $\gf$ on $I_j$ is attained at $\tp_j\in I_j$
(choosing the largest maximum point if it is not unique), then, 
at least for large $j$, $\tp_j$ is in the interior of $I_j$, so $\tp_j$ is
a \lmp{} of $\gf$.
Further, as $j\to\infty$, $\gf(\tp_j)\to\max\psi>\hpsi(0)=1/\ln 2$ while
$\gl\=\E D=\sum_i q_i\gl_i=\sum_1^\infty 2^{-i}=1$, so
$\gf(\tp_j)>\gl$ for large $j$.
Moreover, $\tp_j/2\in I_{j+1}$, and 
since \eqref{magnus} implies
\begin{equation}\label{emma}
  \gf(p/2)=\sumii f(2^{i-1}p)
=\sumi f(2^{i}p)>\gf(p),
\qquad p>0,
\end{equation}
thus
$\gf(\tp_j)<\gf(\tp_j/2)\le\gf(\tp_{j+1})$. It follows that if 
$p\in I_i$ for some $i< j$, then $\gf(p)\le\gf(\tp_i)<\gf(\tp_j)$.
Consequently, for large $j$ at least, $\tp_j$ is a bad \lmp{}, 
and thus there is a phase
transition at
$\pi_j\=\gl/\gf(\tp_j)\in(0,1)$.
This shows that there is an infinite sequence of (first-order)  phase
transitions. 

Further, in this example $\gf$ is bounded
(with $\sup\gf=\max\psi$), and thus $\pic>0$.
Since, by \eqref{emma},  $\sup\gf$ is not attained, this is an example
where the \phtr{} at $\pic$ is continuous and not first-order; simple
calculations show that as $\pi\downto\pic$,
$\pmax(\pi)=\Theta(\pi-\pic)$
and $h_1(\pmax(\pi))=\Theta((\pi-\pic)^2)$. 

Because of the exponential decrease of $|\Gamma(z)|$ on the
imaginary axis, $|\hpsi(n)|$ is very small for $n\neq0$; we have
$|\hpsi(\pm1)|\approx 0.78\cdot 10^{-6}$ and the others much smaller,
so $\psi(x)$ deviates from its mean $\hpsi(0)=1/\ln2\approx1.44$ by
less than $1.6\cdot10^{-6}$. The oscillations of $\psi$ and $\gf$ are
thus very small and hard to observe numerically or graphically unless a large
precision is used. (Taking \eg{} $\gl_i=10^i$ yields larger oscillations.)

Note also that in this example, $\gf$ is not continuous at $p=0$;
$\gf(p)$ is bounded but does not converge as $p\downto0$.
Thus $h$ is not analytic at $p=0$.
\end{example}

\begin{example}\label{Einf2}
  Taking $\gl_i=2^i$ as in \refE{Einfty} but modifying $q_i$ to
  $2^{(\eps-2)i}$ for some small $\eps>0$, similar calculations show
  that $p^\eps\gf(p)=\psi_\eps(p)+O(p)$ for a non-constant
function $\psi_\eps$  with multiplicative period 2, and it follows
  again that, at least if $\eps$ is small enough, there is an infinite
  number of phase transitions. In this case, $\gf(p)\to\infty$ as
  $p\to0$, so $\pic=0$.
\end{example}

Since $\gf$ is analytic on $\ooi$, if there is an infinite number of
bad \crp{s}, then we may order them (and 1, if $\gf'(1)\le0$ and
$\gf(1)=\gl$) in a decreasing sequence $\tp_1>\tp_2>\dots$, with
$\tp_j\to0$. It follows from the definition of bad \crp{s} that then
$\gf(\tp_1)<\gf(\tp_2)<\dots$, and 
$\sup_{0<p\le1}\gf(p)=\sup_j\gf(\tp_j)=\lim_{j\to\infty}\gf(\tp_j)$.
Consequently, if there is an infinite number of phase transitions,
they occur at $\set{\pi_j}_1^\infty\cup\set\pic$ for some decreasing sequence
$\pi_j\downto\pic\ge0$.

\begin{example}\label{E2**i}
We can modify \refand{Einfty}{Einf2} and
consider 
random graphs where all vertex degrees $d_i$ are powers of 2; thus
$D$ has support on \set{2^i}. If we choose $p_{2^i}\sim 2^{-2i}$ or
$p_{2^i}\sim 2^{(\eps-2)i}$  suitably, 
the existence of infinitely many phase transitions follows by
calculations similar to the ones above. (But the details are a little
more complicated, so we omit them.) 
A similar example concentrated on \set{10^i} is shown in
\refF{F10**i}.
\end{example}

\begin{example}\label{EN}
  We may modify \refE{Einfty} by conditioning $D$ on $D\le M$ for some
  large $M$. If we denote the corresponding $h$ and $\gf$ by $h_M$
  and $\gf_M$, it is easily seen that $h_M\to h$ unformly on \oi{} as
  $M\to\infty$, and thus $\gf_M\to\gf$ uniformly on every interval
  $[\gd,1]$.
It follows that if we consider $N$ bad local maximum points of $\gf$,
  then there are $N$ corresponding bad \lmp{s} of $\gf_M$ for large
  $M$, and thus at least $N$ phase transitions.
This shows that we can have any finite number of phase transitions
  with degree sequences $\dnn$ where the degrees are uniformly bounded.
(\refE{ED2} shows that we cannot have infinitely many phase
  transitions in this case.)
\end{example}

\begin{example}\label{ED2}
In Examples \refand{Einfty}{Einf2} with infinitely many phase transitions, we
have $\E D^ 2=\sum_i q_i(\gl_i^2+\gl_i)=\infty$. This is not a
coincidence; in fact, we can show that:
\emph{If\/ $\E D^2<\infty$, then the \kcore{} has only finite number of phase
  transitions.} 

This is trivial for $k=2$, when there never is more than one phase
transition. Thus, assume $k\ge3$.
If $k=3$, then, \cf{} \eqref{h2},
  \begin{equation*}
	\begin{split}
\gf(p)&=h(p)/p^2
=\sum_{l\ge3} lp_l\bigpar{p-p(1-p)^{l-1}-(l-1)p^2(1-p)^{l-2}}p\qww	  
\\&
=\sum_{l\ge3} lp_l\Bigpar{\frac{1-(1-p)^{l-1}}{p}-(l-1)(1-p)^{l-2}}.	  
	\end{split}
  \end{equation*}
Each term in the sum is non-negative and bounded by
$lp_l\bigpar{1-(1-p)^{l-1}}/p\le lp_l(l-1)$, and as $p\to0$ it
converges to $lp_l(l-1-(l-1))=0$. Hence, by dominated convergence,
using the assumption
$\sum p_ll(l-1)=\E D(D-1)<\infty$, we have $\gf(p)\to0$ as
$p\to0$. For $k>3$, $h(p)$ is smaller than for $k=3$ (or possibly the
same), so we have the same conclusion.
Consequently, 
$\gf$ is continuous on $\oi$ and
has a \gmp{} $p_0$ in \ooi. Every bad \crp{} has to belong to
$[p_0,1]$. Since $\gf$ is analytic on
$[p_0,1]$, it has only a finite number of \crp{s} there
(except in the trivial case $\gf(p)=0$), and thus there is only a
finite number of phase transitions.
\end{example}

\begin{example}
  \label{Esecond}
We give an example of a continuous (not first-order) \phtr, letting
$D$ be a mixture as in \refE{Einfty} with two components and carefully
chosen weights $q_1$ and $q_2$.

Let $f$ be as in \refE{Einfty} and note that $f'(x)\sim1/2$ as $x\to0$
and $f'(x)\sim-x^{-2}$ as $x\to\infty$.
Hence, for some $a,A\in\ooo$, $\tfrac14<f'(x)<1$ for $0<x\le 4a$ and
$\tfrac12x\qww\le-f'(x)\le2x\qww$ for $x\ge A$.
Let $f_1(x)\=f(Ax)$ and $f_2(x)\=f(ax)$. Then, $f_1'(x)<0$ for
$x\ge1$. Further, if $g(x)\=f_2'(x)/|f_1'(x)|=(a/A)f'(ax)/|f'(Ax)|$,
then 
\begin{align*}
  g(1)&=\frac {af'(a)}{A|f'(A)|}
<\frac{a}{A\cdot A\qww/2}=2aA,
\\
  g(4)&=\frac {af'(4a)}{A|f'(4A)|}
>\frac{a/4}{A\cdot2 (4A)\qww}=2aA,
\end{align*}
and thus $g(1)<g(4)$.
Further, if $x\ge A/a$, then $f_2'(x)<0$ and thus $g(x)<0$.
Consequently,
$\sup_{x\ge1}g(x)=\max_{1\le x\le A/a}g(x)<\infty$,
and if $x_0$ is the point where the latter maximum is attained
(choosing the largest value if the maximum is attained at several
points), then $1<x_0<\infty$ and $g(x)<g(x_0)$ for $x>x_0$.
Let $\gb\=g(x_0)$ and
\begin{equation}
  \label{m1}
\psi(x)\=\gb f_1(x)+f_2(x)
=\gb f(Ax)+f(ax).
\end{equation}
Then $\psi'(x)\le0$ for $x\ge1$, $\psi'(x_0)=0$ and $\psi'(x_0)<0$ for
$x>x_0$. 

Let $b$ be large, to be chosen later, and let $D$ be as in
\refE{Einfty} with 
$q_1\=\gb a^2/(\gb a^2+A^2)$, $q_2\=1-q_1$, $\gl_1\=bA$, $\gl_2\=ba$.
Then, by \eqref{gf1},
\begin{equation}
  \label{m2}
\gf(p)
=
q_1(bA)^2f(bAp)+q_2(ba)^2f(bap)
=\frac{b^2a^2A^2}{\gb a^2+A^2}\psi(bp).
\end{equation}
Hence, $\gf'(x_0/b)=0$,
$\gf'(x)\le0$ for $x\ge1/b$ and $\gf'(x)<0$ for $x>x_0/b$.
Consequently, $x_0/b$ is a \crp{} but not a \lmp. Furthermore,
$
\gf(x_0/b)=q_1b^2A^2f(Ax_0)+q_2b^2a^2f(ax_0)
$
and
$\gl\=\E D=q_1\gl_1+q_2\gl_2=b(q_1A+q_2a)
$; hence, if $b$ is large enough, then $\gf(x_0/b)>\gl$. We choose $b$
such that this holds and $b>x_0$;
then $\tp\=x_0/b$ is a bad \crp{} which is an inflection point and not
a \lmp. Hence there is a continuous \phtr{} at
$\tpi\=\gl/\gf(\tp)\in(\pic,1)$. 

We have $\gf'(\tp)=\gf''(\tp)=0$; we claim that, at least if $A$ is
chosen large enough, then $\gf'''(\tp)\neq0$. This implies that, for
some $c_1,c_2,c_3>0$, 
$\gf(p)-\gf(\tp)\sim-c_1(p-\tp)^3$ as $p\to\tp$, and 
$\pmax(\pi)-\pmax(\tpi)\sim c_2(\pi-\tpi)^{1/3}$
and $h_1(\pmax(\pi))-h_1(\pmax(\tpi))\sim c_3(\pi-\tpi)^{1/3}$
as $\pi\to\tpi$, so the critical exponent at $\tpi$ is $1/3$.

To verify the claim, note that if also $\gf'''(\tp)=0$, then by
\eqref{m2} and \eqref{m1},
$\psi'(x_0)=\psi''(x_0)=\psi'''(x_0)=0$,, and thus
\begin{equation}
  \label{m3}
\gb A^jf^{(j)}(Ax_0)+ a^jf^{(j)}(ax_0)=0,
\qquad j=1,2,3.
\end{equation}
Let $x_1\=Ax_0$ and  $x_2\=ax_0$. Then $x_1\ge A$, and $f'(x_2)>0$ so
$x_2\le C$ for some $C$.
Further, \eqref{m3} yields
\begin{equation}
\label{m4}
  \frac{x_2f''(x_2)}{f'(x_2)}
=
  \frac{x_1f''(x_1)}{f'(x_1)}
\qquad\text{and}\qquad
  \frac{x_2^2f'''(x_2)}{f'(x_2)}
=
  \frac{x_1^2f'''(x_1)}{f'(x_1)}
.
\end{equation}
Recall that $x_1$ and $x_2$ depend on our choices of $a$ and $A$, and
that we always can decrease $a$ and increase $A$. Keep $a$ fixed and
let $A\to\infty$ (along some sequence). Then $x_1\to\infty$ but
$x_2=O(1)$, so by selecting a subsequence we may assume $x_2\to
y\ge0$.
As $x\to\infty$, $f'(x)\sim-x\qww$,
$f''(x)\sim 2x^{-3}$, and $f'''(x)\sim -6x^{-4}$.
Hence, if \eqref{m4} holds for all large  $A$ (or just a sequence
$A\to\infty$), we obtain by taking the limit
\begin{equation*}
  \frac{yf''(y)}{f'(y)}
=
\lim_{x\to\infty}
  \frac{xf''(x)}{f'(x)}
=-2
\qquad\text{and}\qquad
  \frac{y^2f'''(y)}{f'(y)}
=
\lim_{x\to\infty}
  \frac{x^2f'''(x)}{f'(x)}
=6
.
\end{equation*}

Finally, let $F(x)\=xf(x)=1-(1+x)e^{-x}$.
Then $F''(y)=yf''(y)+2f'(y)=0$ and
$F'''(y)=yf'''(y)+3f''(y)=6y\qw f'(y)-6y\qw f'(y)=0$.
On the other hand, $F'(x)=xe^{-x}$, $F''(x)=(1-x)e^{-x}$,
$F'''(x)=(x-2)e^{-x}$,  so there is no solution to
$F''(y)=yF'''(y)=0$. This contradiction finally proves that
$\gf'''(\tp)\neq0$, at least for large $A$.
\end{example}

\section{Bootstrap percolation in random regular graphs}\label{Sbootstrap}

\emph{Bootstrap percolation} on a graph $G$ is a process that can be
regarded as a model for the spread of an infection.
We
start by infecting a subset $\iio$ of the vertices; 
typically we let $\iio$ be a random subset of the vertex set $V(G)$
such that each
vertex is infected with some given probability $q$, independently of
all other vertices, but other choices are possible, including 
a deterministic choice of $\iio$.
Then, 
for some given threshold $\ell\in\bbN$,
every uninfected 
vertex that has at least $\ell$ infected
neighbours becomes infected. (Infected vertices stay infected; they
never recover.) This is repeated until there are no further infections.
We let 
$\iif=\iifl$ be the final set of infected vertices.
(This is perhaps not a good model for infectious diseases, but may be
reasonable as a model for the spread of rumors or beliefs: 
you are
skeptical the first time you hear something but get convinced the
$\ell$th time.)

Bootstrap percolation is more or less the opposite to taking the
$k$-core. For regular graphs, there is an exact correspondence: it is
easily seen that if the common vertex degree in $G$ is $d$, 
then the set
$V(G)\setminus\iifl$ 
of finally uninfected vertices
equals the $(d+1-\ell)$-core of 
the set
$V(G)\setminus\iio$ 
of initially uninfected vertices.
Furthermore, 
if the initial infection is random, with vertices infected
independently with a common probability $q$, then
the initial infection can be seen as a site
percolation, where each vertex remains uninfected with probability
$\pi=1-q$.
Consequently, in this case 
we obtain results on the size of the final uninfected
set from \refT{TKv}, taking $k=d-\ell+1$ and $\pi=1-q$.

Bootstrap percolation on the random regular graph \gndq{} with fixed
vertex degree $d$ was studied by \citet{BalPitt}.
We can recover a large part of their results from \refT{TKv}.
We have, as just said, $k=d-\ell+1$ and $\pi=1-q$. Moreover, all
degrees $d_i\nn=d$; hence the definitions in \refS{Sdegrees} yield
$n_j=n\gd_{jd}$, $p_j=\gd_{jd}$, $\xD_n=d$, $D=d$ and $\gl=\E D=d$.
\refCond{C1} is satisfied trivially.
Furthermore, \eqref{h} and \eqref{h1} yield, since $D_p\sim\Bi(d,p)$,
\begin{equation*}
  \begin{split}
h(p)
&=
\sum_{j=k}^d j \bxxx djp
=
\sum_{j=k}^d j\binom dj p^j(1-p)^{d-j}
=
\sum_{j=k}^d dp\binom {d-1}{j-1} p^{j-1}(1-p)^{d-j}
\\&
=
dp\P\bigpar{\Bi(d-1,p)\ge k-1}	
=
dp\P\bigpar{\Bi(d-1,p)\ge d-\ell}	
\\&
=
dp\P\bigpar{\Bi(d-1,1-p)\le \ell-1}	
  \end{split}
\end{equation*}
and
\begin{equation*}
  \begin{split}
h_1(p)
&=
\P\bigpar{\Bi(d,p)\ge k}	
=
\P\bigpar{\Bi(d,p)\ge d-\ell+1}	
=
\P\bigpar{\Bi(d,1-p)\le \ell-1}	
.
  \end{split}
\end{equation*}
Consequently, \eqref{pick} yields
\begin{equation*}
  \pic
\=
\inf_{0<p\le1}\frac{d p^2}{h(p)} 
= 
\inf_{0<p\le1}\frac{p}{\P\bigpar{\Bi(d-1,1-p)\le \ell-1}}.
\end{equation*}
We define $\qc\=1-\pic$, and \refT{TKv} translates as follows.
(Recall that we have proven \refT{TKv} for the random multigraph
$\gndqx$, but as said in the introduction, the result holds for the
simple random graph $\gndq$ by a standard conditioning.)

\begin{theorem}[\cite{BalPitt}]
    \label{Tboot}
Let $d$, $\ell$ and $q\in\oi$ be given  with $1\le\ell\le d-1$.
Consider bootstrap percolation on the random $d$-regular graph
$\gndq$,
with threshold $\ell$ and vertices initially infected randomly
with probability $q$, independently of each other. 
Let
\begin{equation}\label{qc}
  \qc=\qcl
\=
1-\inf_{0<p\le1}\frac{p}{\P\bigpar{\Bi(d-1,1-p)\le \ell-1}} .
\end{equation}
  \begin{romenumerate}
\item
If $q>\qc$, then $|\iif|=n-\opn$.
Furthermore, 
if $l\le d-2$ then  \whp{} $|\iif|=n$, \ie, all vertices 
eventually become infected.
\item
If $q<\qc$, then \whp{} a positive proportion of the vertices remain
uninfected, More precisely,
if\/ $\pmax=\pmax(q)$ is the largest $p\le1$ such that
$\P\bigpar{\Bi(d-1,1-p)\le \ell-1}/ p=(1-q)\qw$, 
then
\begin{align*}
|\iif|/n\pto 1-(1-q) \P\bigpar{\Bi(d,1-\pmax)\le\ell-1}
<1.
\end{align*}
\end{romenumerate}
\end{theorem}

\begin{proof}
It remains only to show that in case (ii),
$\pmax$ is not a \lmp{} of 
$\bgf(p)\=h(p)/(dp^2)=\P\bigpar{\Bi(d-1,1-p)\le \ell-1}/ p$.
(In the notation of \refS{Score}, $\bgf(p)=\gf(p)/\gl$.)
In fact, some simple calculus shows, see 
\cite[\S3.2, where $R(y)=\bgf(y)\qw$]{BalPitt}  for details,
that the function $\bgf$
is unimodal when
$\ell<d-1$ and decreasing when $\ell=d-1$; thus there is no \lmp{} 
when $\ell=d-1$, and otherwise 
the only \lmp{} is
the \gmp{} $p_0$ with $\bgf(p_0)=\pic\qw=(1-\qc)\qw$.
(It follows also \cite{BalPitt} that 
the equation
$\bgf(p)=(1-q)\qw$ 
in (ii) has exactly two roots for every $q<\qc$
when $\ell<d-1$ and one when $\ell=d-1$.)
\end{proof}

\begin{remark}
  The case $\ell=1$ ($k=d$) is rather trivial;
in this case, $\iif$ is the union of all components of $\gndq$ that
contain at least one initially infected vertex. If further
$d\ge3$, then
$\gndq$ is \whp{} connected, and thus any non-empty $\iio$ \whp{}
yields $|\iif|=n$. (The case $d=2$ is different but also simple:
$G(n,2)$ consists of disjoint cycles, and only a few small cycles will
remain uninfected.)
\end{remark}

Actually, \citet{BalPitt} study primarily the case when the 
initially infected set $\iio$ is deterministic; they then derive
the result above for a random $\iio$ by conditioning on $\iio$.
Thus, assume now that $\iio$ is given, with $|\iio|=m$.
 (For $\gndq$, because
of the symmetry, it does not matter whether we remove a specified set
of $m$ vertices or a uniformly distributed random set with $m$ vertices.)
Assuming $m\sim nq$, we have the same
results in this case, see \refR{Rfixed}; indeed, the proof is slightly
simpler
since the use of the law of large numbers in \refSS{SSvertex} is
replaced by the 
obvious  $\ny_d=n-m$, $\ny_1=\nr=dm$.

\begin{theorem}
 \refT{Tboot} remains valid if the initially infected set is
any given set with $m=m(n)$ vertices, where $m/n\to q$.
\end{theorem}

\begin{remark}
As in \refR{Rjumps}, it is also possible to study the threshold in
greater detail by allowing $q$ to depend on $n$.
If we assume $\ell\le d-2$ (\ie, $k\ge3$) and $q=q(n)\to\qc$ defined
by \eqref{qc}, then \citet[Theorem 3.5]{SJ196} applies
and implies the following, also proved by
\citet{BalPitt} by different methods.
\begin{theorem}[\cite{BalPitt}]
Consider bootstrap percolation on $\gndq$ with $\ell\le d-2$, and
assume
that the set $\iio$ of initially infected vertices either is 
deterministic with $|\iio|=q(n)n$ or random, with each vertex
infected with probability $q(n)$.
  \begin{romenumerate}
\item
If $q(n)-\qc\gg n\qqw$, then \whp{} $|\iif|=n$, \ie, all vertices
become infected.	
\item
If $\qc-q(n)\gg n\qqw$, then \whp{} $|\iif|<n$ and,
moreover,
\begin{equation*}
|\iif|=h_1\bigpar{\pmax(q(n))}n+O_p\bigpar{n\qq|q(n)-\qc|\qqw}.  
\end{equation*}
  \end{romenumerate}
\end{theorem}

\citet[Theorem 3.5]{SJ196} is stated for random multigraphs, and for
$\gndqx$ it yields further an asymptotic normal distribution in case
(ii), as well as a precise result for $\P(|\iif|=n)$ in the transition
window $q(n)-\qc=O(n\qqw)$.
The latter result can easily be transformed into the following
analogue of \cite[Theorem 1.4]{SJ196}; the asymptotic variance $\gss$
is given by explicit but rather complicated formulas in \cite{SJ196}.

\begin{theorem}\label{TM}
Assume $\ell\le d-2$.
Infect (from the outside) the vertices in $\gndqx$ 
one by one in random order, letting the infection spread as above to
every vertex having at least $\ell$ infected neighbours,
and let $M$ be the number of externally infected vertices required
to make $|\iif|=n$. Then $(M-n\qc)/n\qq\dto N(0,\gss)$, with $\gss>0$.
\end{theorem}

Presumably, the same results hold for $\gndq$, but technical
difficulties have so far prevented a proof, cf.\ \cite{SJ196}. In
any case, it follows from \refT{TM} that the size of
the transition window is $O(n\qqw)$ for $\gndq$ too, and not smaller.
\end{remark}

\appendix

\section{The \kcore{} and branching processes}\label{App}

We give a precise statement of the relation between \refP{PTK} and
branching processes. This can be seen heuristically from the
branching process approximation of the local exploration process,
but as said above, we do not attempt to make this approximation
rigorous; instead we compare the quantities in \refP{PTK} with
branching process probabilities.

\begin{theorem}
Let $\cX$ be a Galton--Watson branching process with offspring
distribution $\D$ and starting with one individual $o$, and let $\cXx$ be
the modified branching process where the root $o$ has offspring
distribution $D$ but everyone else has offspring distribution $\D$.
We regard these branching processes as (possibly infinite) trees with
root $o$. Further, let $\tk$ be the infinite rooted tree where each
node has $k-1$ children, and let $\tkx$ be the infinite rooted
$k$-regular tree where the root has $k$ children and everyone else
$k-1$.

Then $\pmax=\P(\cX\supseteq\tk)$, the probability
that $\cX$ contains a rooted copy of $\tk$ (\ie, a copy of $\tk$ with
root $o$)
and 
$h_1(\pmax)=\P(\cXx\supseteq\tkx)$, the probability
that $\cXx$ contains a rooted copy of $\tkx$.
\end{theorem}

Hence, by \refP{PTK},
the probability that a random vertex belongs to the \kcore,
which is $\E(v(\crkx)/n)$, converges to the probability
$\P(\cXx\supseteq\tkx)$, the probability that the branching process 
approximating the local structure at a random vertex contains the
infinite $k$-regular tree $\tkx$.
Similarly, the probability that a random
edge belongs to the $\crkx$, which is
$\sim\E(e(\crkx)/(n\gl/2))$, converges to
$h(\pmax)/\gl=\pmax^ 2=\P(\cX\supseteq\tk)^ 2$, which can be
interpreted as the probability that both endpoints of a random edge
grow infinite $k$-regular trees in the branching process approximation.

\begin{proof}
Let $\tkn$ be the subtree of $\tk$ consisting of
all nodes of height $\le n$, \ie, the rooted tree of height $n$ where
each node of height $<n$ has $k-1$ children, and let $q_n$ be the
probability that $\cX$ contains a copy of $\tkn$. Thus, $q_0=1$, and
for $n\ge0$, $q_{n+1}$ is the probability that the root $o$ has at
least $k-1$ children that each is the root of a copy of $\tkn$ in the
corresponding subtree of $\cX$; let us call such children \emph{good}.
By the branching property, the subtrees rooted at the children of $o$
are independent copies of $\cX$: thus the probability that a given
child is good is $q_n$, and the number of good children of $o$ has
the tinned distribution $\D_{q_n}$. Hence,
\begin{equation*}
  \begin{split}
	q_{n+1}
&=\P(\D_{q_n}\ge k-1)
=\sum_{d=k}^ \infty \P(\D=d-1) \sum_{l=k}^ \infty\P\bigpar{\Bi(d-1,q_n)=l-1}
\\&
=\sum_{d\ge k} \sum_{l\ge k} \frac{d}{\gl}\P(D=d)\frac{l}{dq_n}
  \P\bigpar{\Bi(d,q_n)=l}
\\&
=\frac1{\gl q_n} \sum_{l\ge k}l\P(D_{q_n}=l)
=\frac1{\gl q_n} h(q_n).
  \end{split}
\end{equation*}
Since $x\mapsto h(x)/(\gl x)$ is increasing (\eg{} by the same
calculation)
and $1=q_0\ge q_1\ge
\dots$, it follows that $q_n$ decreases to the largest root $\pmax$ of 
$\frac1{\gl q} h(q)=q$ in $\oi$.
On the other hand, the events $\cE_n\=\set{\cX\supseteq \tkn}$ are
decreasing, $\cE_1\supseteq\cE_2\supseteq\dotsm$, and $\bigcap_n\cE_n$ is, by a
compactness argument, equal to the event \set{\cX\supseteq\tk}. Hence, 
$\P(\cX\supseteq\tk)=\lim_n q_n=\pmax$.

Similarly, $\cXx$ contains a rooted copy of $\tkx$ if and only if the
root $o$ has at least $k$ good (now with $n=\infty$)
children. We have shown that each child
is good with probability $\pmax$, and thus the number of good children
has the thinned distribution $D_{\pmax}$; hence
$\P(\cXx\supseteq\tkx)=\P(D_{\pmax}\ge k)=h_1(\pmax)$.
\end{proof}

\begin{remark}
  When $k=2$, $\cT_2$ is just an infinite path, and thus
  $\pmax=\P(\cX\supseteq\cT_2)$ is just the survival probability $\rho$ of the
  branching process $\cX$, as observed algebraically in \refR{Rk=2}.
Hence the thresholds for 2-core and giant
  component coincide, for any of our percolation models.
Moreover, we see that if $v$ is a random vertex, the events
``$v$ is in a giant component'' and ``$v$ is in the  2-core'' are
  approximated by ``the root $o$ in $\cXx$ has at least one child with
  infinite progeny''
and ``the root $o$ in $\cXx$ has at least two children with
  infinite progeny'', respectively, which again shows the close
  connection between these properties.
\end{remark}

\begin{ack}
Part of this research was done during visits to
to the University of Cambridge and Trinity College in 2007,
the 11th Brazilian School of Probability in Maresias, August 2007, and
to the Isaac Newton Institute in Cambridge and Churchill College in
2008.
I thank 
Louigi Addario-Berry,
Malwina Luczak,
Rob Morris
and James Norris
for stimulating discussions and helpful comments.
\end{ack}

\newcommand\AAP{\emph{Adv. Appl. Probab.} }
\newcommand\JAP{\emph{J. Appl. Probab.} }
\newcommand\JAMS{\emph{J. \AMS} }
\newcommand\MAMS{\emph{Memoirs \AMS} }
\newcommand\PAMS{\emph{Proc. \AMS} }
\newcommand\TAMS{\emph{Trans. \AMS} }
\newcommand\AnnMS{\emph{Ann. Math. Statist.} }
\newcommand\AnnPr{\emph{Ann. Probab.} }
\newcommand\CPC{\emph{Combin. Probab. Comput.} }
\newcommand\JMAA{\emph{J. Math. Anal. Appl.} }
\newcommand\RSA{\emph{Random Struct. Alg.} }
\newcommand\ZW{\emph{Z. Wahrsch. Verw. Gebiete} }
\newcommand\DMTCS{\jour{Discr. Math. Theor. Comput. Sci.} }

\newcommand\AMS{Amer. Math. Soc.}
\newcommand\Springer{Springer-Verlag}
\newcommand\Wiley{Wiley}

\newcommand\vol{\textbf}
\newcommand\jour{\emph}
\newcommand\book{\emph}
\newcommand\inbook{\emph}
\def\no#1#2,{\unskip#2, no. #1,} 
\newcommand\toappear{\unskip, to appear}

\newcommand\webcite[1]{
   \penalty0
\texttt{\def~{{\tiny$\sim$}}#1}\hfill\hfill}
\newcommand\webcitesvante{\webcite{http://www.math.uu.se/~svante/papers/}}
\newcommand\arxiv[1]{\webcite{arXiv:#1.}}

\def\nobibitem#1\par{}


\begin{thebibliography}{99}

\bibitem[Balogh and Pittel(2007)]{BalPitt} 
J. Balogh, B. G. Pittel, 
Bootstrap percolation on the random regular graph.
\RSA
\vol{30}  (2007),  no. 1-2, 257--286. 


\bibitem[Bollob\'as(1984)]{Bollobas84}
B. Bollob\'as, 
The evolution of sparse graphs.  
\inbook{Graph theory and Combinatorics} (Cambridge, 1983),  35--57, 
Academic Press, London, 1984.


\bibitem[Bollob\'as(2001)]{Bollobas}
B. Bollob\'as, 
\book{Random Graphs}, 2nd ed., Cambridge Univ. Press,
Cambridge, 2001.

\bibitem[Britton, Janson and Martin-L\"of(2007+)]{SJ199}
T. Britton, S. Janson \& A. Martin-L\"of,
Graphs with specified degree distributions, simple epidemics and local
vaccination strategies.
\jour{Advances Appl. Probab.}
\vol{39} (2007), no. 4, 922--948. 

\bibitem{CW}
J. Cain \& N. Wormald,
Encore on cores.
\jour{Electronic J. Combinatorics},
\vol{13}\no1 (2006), R81.

\bibitem[Darling, Levin and Norris(2004)]{DarlingLN}
R. W. R. Darling, D. A. Levin \& J. R. Norris,
Continuous and discontinuous phase transitions in hypergraph
processes. 
\RSA \vol{24} (2004), no. 4, 397--419.

\bibitem[Fountoulakis(2007+)]{Fount}
N. Fountoulakis,
Percolation on sparse random graphs with given degree sequence.
Preprint, 2007.
\arxiv{math/0703269v1} 

\bibitem[Gut(2005)]{Gut}
A. Gut,
\book{Probability: A Graduate Course}.
Springer, New York, 2005.


\bibitem[Janson(2007+)]{SJ195}
S. Janson,
The probability that a random multigraph is simple. 
Preprint, 2006.
\arxiv{math.CO/0609802}

\bibitem[Janson and Luczak(2007)]{SJ184}
S. Janson \& M. J. Luczak,
A simple solution to the $k$-core problem.
\RSA \vol{30} (2007), 
50--62.

\bibitem[Janson and Luczak(2007+)]{SJ196}
S. Janson \& M. J. Luczak,
Asymptotic normality of the $k$-core in random graphs.
\jour{Ann. Appl. Probab.}, to appear.
\arxiv{math.CO/0612827}


\bibitem[Janson and Luczak(2007+)]{SJ204}
S. Janson \& M. J. Luczak,
A new approach to the giant component problem. 
\RSA\unskip, to appear.
\arxiv{0707.1786v1}  

\bibitem[Kallenberg(2002)]{Kallenberg}
O. Kallenberg,
\emph{Foundations of Modern  Probability},
2nd ed.,
\Springer, New York, 2002.

\bibitem[\L uczak(1991)]{Luczak91}
T. {\L}uczak,  Size and connectivity of the $k$-core of a random
graph, \jour{Discr. Math.} \vol{91} (1991) 61--68.

\bibitem[Molloy and Reed(1995)]{MR95}
M. Molloy \& B. Reed, 
A critical point for random graphs with a given degree sequence,
\RSA \vol6 \no{2--3} (1995), 161--179.

\bibitem[Molloy and Reed(1998)]{MR98}
M. Molloy \& B. Reed, 
The size of the largest component of a random graph on a fixed degree
sequence, \CPC \vol7  (1998), 295--306.

\bibitem[Pittel, Spencer and Wormald(1996)]{PSW} 
B. Pittel, J. Spencer \& N. Wormald, 
Sudden emergence of  a giant $k$-core in a random graph, 
\jour{J. Combin. Theor. Ser. B} \vol{67}
  (1996), 111--151. 

\bibitem[Riordan(2007+)]{Riordan:kcore}
O. Riordan,
The $k$-core and branching processes.
\jour{Combin. Probab. Comput.}, to appear.
Published online 
27 Jun 2007.


\end{thebibliography}
\end{document}